\documentclass[12pt,leqno]{amsart} 
\usepackage[backgroundcolor=yellow,colorinlistoftodos]{todonotes}
\usepackage{fourier}
\usepackage{datetime}
\usepackage{graphicx,float,tabularx}
\usepackage[utf8]{inputenc}
\usepackage{enumitem}   
\usepackage{amsmath,amssymb,amsfonts,mathrsfs,mathtools}
\usepackage{hyperref}
\usepackage{enumitem}
\usepackage[all,2cell]{xy}
\xyoption{arrow}
\SelectTips{cm}{10}

\newdateformat{myformat}{\THEDAY\ \monthname[\THEMONTH] \THEYEAR,
  \currenttime}
\hoffset - 2cm
\makeatletter
\renewcommand\@makefnmark{%
  \hbox{\textsuperscript{\normalfont\color{red}\@thefnmark}}}
\makeatother

\makeatletter
\providecommand\@dotsep{5}
\def\listtodoname{List of Todos}
\def\listoftodos{\@starttoc{tdo}\listtodoname}
\makeatother

\newcommand{\bookmarktodo}[2][]{\global\advance\count66 by 1
\footnote{\color{red}\textbf{#2}}%
  \addcontentsline{tdo}{todo}{\protect{
    {\color{red}[\the\count66]}\hskip 1em
    \parbox{14cm}{#2}%
  }}%
}

\catcode`\@=11 
\def\@evenfoot{\rule{0pt}{20pt}[\myformat\today] \hfill [{\tt \jobname.tex}]}
\def\@oddfoot{\rule{0pt}{20pt}{[\tt \jobname.tex}]\hfill [\myformat\today]}
\catcode`\@=13

\usepackage[object=vectorian]{pgfornament}

\usepackage{bbold,euscript}

\usepackage{quiver}

\usepackage{tikz}
\usetikzlibrary{arrows.meta}
\usetikzlibrary{decorations.markings}

\newtheorem{theorem}{Theorem}
\newtheorem{corollary}[theorem]{Corollary}

\newtheorem{lemma}[theorem]{Lemma}
\newtheorem{proposition}[theorem]{Proposition}

\theoremstyle{definition}
\newtheorem{example}[theorem]{Example}

\newtheorem{remark}[theorem]{Remark}
\newtheorem{definition}[theorem]{Definition}

\def\Xrightarrow#1{\stackrel{#1}\longrightarrow}
\def\dec{{\rm dec}}
\def\Dec{\mathbb D}
\def\spicka{\raisebox{.8em}{\rotatebox{-90}{$\triangleright$}}}
\def\O{\tt O}
\def\Fib{\triangleright}
\def\Fin{\tt Fin}
\def\T{\mathbb{T}}
\def\Ar{\mathbb{A}}
\def\Com{\mathbb{C}}
\def\ee{\varepsilon}

\def\redukce#1{\vbox to .3em{\vss\hbox{#1}}}
\def\tecka{\, . \hspace {-.8em}}
\def\Xrightarrow#1{\stackrel {#1}\longrightarrow}

\def\dminus{{d^\Ar_{-1}}}
\def\sminus{{s^\Ar_{-1}}}
\def\Acoalg{\Cat_\Ar}
\def\oA{{\overline {\Ar}}}
\def\S{\between}
\def\Rada#1#2#3{#1_{#2},\dots,#1_{#3}}
\def\rada#1#2{{#1,\ldots,#2}}
\def\ttN{{\mathbb N}}
\def\ttNw{{\mathbb N}^w}
\def\ttNW{{\widetilde{\mathbb N}}^w}
\def\ttn{{\tt n}}

\def\id{{\mathbb 1}}

\def\F{{\EuScript F}}

\newcommand{\C}{{\tt C}}
\newcommand{\D}{\tt{D}}
\newcommand{\Cat}{{\tt Cat}}

\newcommand{\uu}{\mathbb{1}}

\newcommand{\NN}{\mathbb{N}}

\newcommand{\K}{\EuScript{K}}

\newcommand{\x}{\times}

\renewcommand{\epsilon}{\varepsilon}
\renewcommand{\phi}{\varphi}
\renewcommand{\rho}{\varrho}
\newcommand{\jednajedna}{\mathbf{1}}
\newcommand{\dvedve}{\mathbf{2}}

\title[Kernels, algebras, d\'ecalage, supercoherence]
{Kernels, lax algebras, d\'ecalage, and supercoherence}
\author[M.\ Markl, D.\ Trnka]{Martin Markl and Dominik Trnka}
\address{Martin~Markl: Institute of Mathematics, The Czech
  Academy of Sciences, {\v Z}itn{\'a} 25, 110 00 Praha 1, The Czech Republic}
\address{Dominik~Trnka: Institute of Mathematics, University of Technology,  Technick\'a 2896, 616 69 Brno, The Czech Republic}

\catcode`\@=11
\email{markl@math.cas.cz; trnka@fme.vutbr.cz}
\@namedef{subjclassname@2010}{%
  \textup{2020} Mathematics Subject Classification}
\catcode`\@=13

\subjclass[2010]{Primary 18C15, 18N50, 18M60}
\keywords{kernels, supercoherence, lax algebra, operadic category.}

\begin{document}

\baselineskip 17pt plus 2pt minus 1pt

\begin{abstract}
We prove that a pointed category has kernels if and only if it is a
lax algebra for the arrow $2$-monad, and that this holds if and only
if it is the d\'ecalage of a supercoherent structure. We will then
interpret categories with kernels as the sought-after weak version of
unary operadic categories.
\end{abstract}

\thanks{M.M. supported by  RVO: 67985840. D.T. supported by FSI-S-26-8958.}
\maketitle

\tableofcontents

\section*{Introduction}

A {\/\em   kernel\/} of a morphism $f: a \to b$ 
of a pointed category is a morphism
$\kappa :k \to a$ such that the composition $f \kappa : k \to b$ is a
null morphism, and such that $\kappa$ is terminal among all morphisms
with this property.
A category with kernels is a category in which every morphism has a
kernel. Examples are abundant: each Abelian category has kernels as
does the category of groups, \&c. 

We first prove that pointed categories are
coalgebras for the $2$-comonad $\Ar$ that sends a category~$\C$ to
its category of arrows  $\Ar(\C)$. 
We then identify categories with kernels with  lax 
algebras for the induced $2$-monad $\overline \Ar$ on the
category of $\Ar$-coalgebras.

John F.~Jardine, in his remarkable 1991 article~\cite{Jardine},
introduced supercoherent structures as a special kind of lax simplicial objects
$X_\bullet$ in the category $\Cat$ of categories. Our second main
result states that a pointed category $\C$ is a category with kernels
if and only if its nerve $\ttN_\bullet(\C)$ is the d\' ecalage
$\dec_\bullet(X)$ of a supercoherent structure $X_\bullet$. Categories with
kernels thus provide another example of supercoherence, 
alongside Jardine's classical example of the bar construction of a 
monoidal~category. 

Leaving out the details, we therefore proved that, for a pointed
category $\C$, the following are equivalent:
\hfill\break 
\noindent
\hphantom{1em} -- the category $\C$ has kernels,
\hfill\break 
\noindent 
\hphantom{1em} --  the category $\C$ is a lax algebra for the arrow $2$-monad $\overline \Ar$, and
\hfill\break 
\noindent 
\hphantom{1em} -- the nerve of  the category 
$\C$ is a d\'ecalage of a supercoherent structure.
\hfill\break \noindent 
The above equivalences combine
Theorems~\ref{theorem:kernels} and~\ref{Jen aby si
  stihla vyridit pas.}, which together constitute the  main result of this work. 

\vskip .05em
\noindent 
{\bf Motivations and inspirations.}
Operadic categories, introduced in~\cite{duo}, are categories in
which every morphism has fibers. 
In unary operadic categories, such a fiber is unique.  From a
conceptual perspective, Abelian categories should provide examples, 
with fibers given by
kernels. However, such a structure does not satisfy
the main axiom of operadic categories, which requires that the kernel
of the induced morphism between the kernels should {\em coincide\/} with 
the kernel of the original morphism;  
in general Abelian categories, these kernels are only
isomorphic. 
This fact, together with standard categorical objection to strict 
equalities, 
suggested the need for a weak version of operadic categories.
Further discussion of this perspective can
be found in Section~\ref{Jarka pojede do Kolumbie, ale Nove Eldorado
  je legenda.}. 

Operadic categories and their relation to other structures have been
the subject of several works. Particularly important for us
was~\cite{GarnerKockWeber}, which brought into focus the d\'ecalage comonad
$\Dec : \Cat \to \Cat$. The authors proved that unary operadic
categories are algebras for the induced monad on the category of
\,$\Dec$-coalgebras. Another source of our inspiration was the
characterization of unary operadic categories as categories whose
nerve is the upper d\'ecalage of a simplicial set; see, for
example,~\cite{blob}.

Our strategy, therefore, was to {\em define} \/ weak operadic
categories either as lax algebras for the induced d\'ecalage monad
$\Dec$, or alternatively as the d\'ecalage of a lax version of a
simplicial set. The difficulty with the former approach is that $\Dec$
is not a $2$-(co)monad, so the notion of lax algebras does not
make sense. It must be replaced by the arrow (co)monad $\Ar$, which is,
ideologically, a \hbox{$2$-completion} of $\Dec$.  As for the latter
approach, replacing simplicial sets by Jardine’s supercoherent
structure requires interpreting the nerve $\ttN_\bullet(\C)$ as a
simplicial object in $\Cat$ rather than as a simplicial set.

This shift in perspective to the $2$-categorical level dramatically
changed the picture. It turned out that the fiber of a morphism in
a ‘weak’ unary operadic category, defined along the above lines, is
automatically equipped with a morphism to its domain, so it
resembles a kernel rather than a fiber. Remarkably, the universal
property of kernels is also satisfied, leading us to the realization that a weak
unary operadic category is in fact a category with kernels! More
details can be found in Section~\ref{Jarka pojede do Kolumbie, ale
  Nove Eldorado je legenda.}.

Our ultimate aim is to formulate a bivariant version of the theory
that accommodates both kernels and cokernels and, together with
additivity, characterizes abelian categories as lax bialgebras for a
$2$-bimonad, or equivalently as a double---upper and lower---d\'ecalage of
a supercoherent structure. This will be the subject of future
work. The first steps along this path have already been taken
in~\cite{bivar}.

\vskip .5em
\noindent 
{\bf Categorical context.}
While in the present article we start with the {\em comonad}
structure of the arrow category functor~$\Ar$ and then work with the
induced monad $\overline{\Ar}$ on the category of $\Ar$-co\-al\-geb\-ras, 
a somewhat dual approach,
based on a natural {\em monad} structure on $\Ar$  was studied 
in~\cite{Tholen}, where $\Ar$-algebras were identified 
with orthogonal factorization systems. 
We acknowledge that Section~\ref{Bures vladne vsem.} 
of this article was inspired by the methods of that paper.

Our work touches on the realm of
property-like structures and (co)lax-idempotent 2-monads, as discussed
for instance in \cite{KellyLack}. A 2-monad $\T$ with unit $\eta$ is
lax-idempotent if, for every algebra
$a\colon \T(A)\to A $, there exists a 2-cell
$\theta \colon {\rm id}_{\T (A)}\Rightarrow \eta_A\circ a$ such that
$({\rm id},\theta)$ is the counit-unit pair  of an adjunction
$a\dashv \eta_A$.  Equivalently,  
given two \hbox{$\T$-algebras $a\colon \T (A )\to A$},
$b\colon \T (B )\to B$ and a morphism $f\colon A \to B$, there exists a
unique 2-cell $\overline{f}$ extending $f$ to a lax morphism of
$\T$-algebras; see \cite[Theorem~6.2]{KellyLack}. In particular, any two
$\T$-algebra structures $a', a''\colon \T A\to A$ on the same object
$A$ are related by a unique $\T$-algebra isomorphism, so algebras for
a lax-idempotent \hbox{$2$-monad} on $\Cat$ are typically categories equipped
with structure satisfying a specified universal property.  A paradigmatic
examples are $2$-monads $\T$ that assigns to each category $\C$ its
free cocompletion under a given class of colimits. A $\T$-algebra is
then a category admitting all such colimits. As we show, for a pointed
category $\C$, an $\overline{\Ar}$-algebra structure corresponds to a
choice of a kernel for each morphism in $\C$. By the universal
property of kernels, any two choices are isomorphic, and
hence so are the corresponding 
algebra~structures.

\vskip .5em
\noindent {\bf Plan of the paper.}
In Section~\ref{Bures vladne vsem.} we introduce the arrow $2$-comonad
$\Ar \colon \Cat \to \Cat$ and analyze its coalgebras. In
Section~\ref{Dominik prileti v pondeli.} we then explicitly describe lax
  algebras for the $2$-monad  $\overline \Ar$ induced by $\Ar$ on the
  category of its coalgebras. Section~\ref{Zacina Zimni skola. Idealni
    snehove podminky ale ja trcim v Bonnu!} contains the first main
  result of the paper, Theorem~\ref{theorem:kernels}, which characterizes
  kernels as normalized lax algebras for the arrow monad $\overline \Ar$.
Section~\ref{Jarka upekla skvely jablecny zavin.} contains auxiliary
material relating the combinatorics of the nerve to the arrow
category; this will be used in Section~\ref{Na tu novou vladu opravdu
  nemam nerv ani jako kategorie.} to describe the nerve of lax algebras.
Section~\ref{Jarka si koupila novy mobil.} contains the second main
result, Theorem~\ref{Jen aby si stihla vyridit pas.}, which
characterizes 
kernels via the d\'ecalage of a
supercoherent structure. The final Section~\ref{Jarka pojede do
  Kolumbie, ale Nove Eldorado je legenda.} interprets categories
with kernels as weak unary operadic~categories. 

\vskip .5em
\noindent {\bf Conventions.}
We will assume that all categories are small relative to a
sufficiently large ambient universe. We will also assume the axiom of
choice, so that we need not distinguish between categories with
kernels and categories with {\em algebraic} kernels,
i.e.~categories in which each morphism has a {\em chosen}
kernel.

\vskip .5em
\noindent {\bf Acknowledgment.}  Both authors are
grateful to the Max-Planck-Institut f\"ur Mathematik in Bonn for its
hospitality and financial support during the last stages of the work.

\section{The arrow comonad}
\label{Bures vladne vsem.}

In the rest of the paper, the symbol $\id$, with or without
subscript, will denote---depending on the context---the identity
endomorphism, functor or transformation.
Let $\jednajedna$ be the terminal category, and denote its unique
object by $0$.  Let $\dvedve$ be the category with two objects $0$,
$1$ and one non-identity morphism $a\colon 0\to 1$. The category
$\dvedve\x\dvedve$ consists of one commutative square
\begin{equation}\label{equation:4-square}
    \raisebox{-1.8em}{$\dvedve\x\dvedve=$ \ }
\xymatrix{(00)\ar[r]^{a1} \ar[d]_{0a}  &(10)  \ar[d]^{1a} 
\\
(01)\ar[r]^{a0} & \, (11).
}
\end{equation}

Denote by $m\colon\dvedve\x\dvedve\to\dvedve$ the functor which sends an
object $(ij) \in \dvedve\x\dvedve$ 
to $s(ij)=\max \{i,j\}$; its value on morphisms is thus
determined. Explicitly, $m(01) = m(10) = m(11) = 1$, $m(00) = 0$, and
$m(a1) = 1 = m(1a)$, $m(a0) = a = m(0a)$.
Consider also the functor $e\colon \jednajedna \to \dvedve$ with
$e(0)=0$.  Assuming the canonical isomorphisms
$(\dvedve\x\dvedve)\x\dvedve\cong \dvedve\x(\dvedve\x\dvedve)$ and
$\dvedve\x \jednajedna\cong\dvedve\cong \jednajedna\x\dvedve$, the diagrams
\[
\xymatrix@C=3em{\dvedve  \x \dvedve  \x \dvedve \ar[r]^{\dvedve \x m}     
\ar[d]_{m \x \dvedve}
&\dvedve  \x \dvedve  \ar[d]^m
\\
\dvedve  \x \dvedve \ar[r]^m   &\dvedve
}
\hspace{2em}
\xymatrix@C=3em{\dvedve  \ar[r]^(.4){\dvedve \x e} \ar[rd]_{\id}  &   
\dvedve \x \dvedve \ar[d]^m
\\
&\dvedve
}
\hspace{2em}
\xymatrix@C=3em{\dvedve  \ar[r]^(.4){e \x \dvedve} \ar[rd]_{\id}  &   
\dvedve \x \dvedve \ar[d]^m
\\
&\dvedve
} 
\]
commute. Hence
$(\dvedve,m,e)$ is a monoid in~$(\Cat,\x,\jednajedna)$. Denote by
$\tau\colon \dvedve\x\dvedve\cong\dvedve\x\dvedve$ the symmetry
isomorphism of the cartesian product in $\Cat$. The equality
$\max \{i,j\}=\max \{j,i\}$ 
translates to 
\begin{equation}
\label{Predvcerem jsem dojel do Bonnu.}
m\tau=m,
\end{equation}
so $\dvedve$ is a strict symmetric, 
strictly associative monoidal category. The monoidal structure of $\dvedve$ leads to a comonad $\Ar $ on $\Cat$,
with
\begin{equation}
\label{Ztratil jsem klicek od kola - porad neco ztracim.}
\Ar (\C):=[\dvedve\to\C], \ \C \in \Cat,
\end{equation}
the category of functors and natural
transformations.

Let us describe the arrow category 2-functor $\Ar \colon \Cat \to \Cat$
in more detail. The objects of $\Ar (\C)$ are morphisms of $\C$ and the
morphisms of $\Ar (\C)$ are commutative squares
in~$\C$. A~morphism 
\[
S=(h_0,h_1)\colon f\longrightarrow g
\] 
in $\Ar (\C)$ will be
depicted as the square
\begin{equation}\label{equation:square}  
    \raisebox{-1.8em}{$S\ =$ \ }
\xymatrix{a \ar[r]^{h_0} \ar[d]_{f}  & c \ar[d]^{g}
\\
b  \ar[r]^{h_1} & d.
}
\end{equation}
By convention, the morphisms in $\Ar (\C)$ always go 
from the left edge of a square to the right edge.
For a functor $F\colon \C\to \D$, the value
$\Ar  (F)\colon \Ar (\C) \to \Ar (\mathtt{D})$ 
is given by post-composition with $F$, that is 
\begin{align*}
\Ar  (F)(f\colon a\to b)&=F(f)\colon F(a) \longrightarrow F(b),
\\
\Ar  (F)\big((h_0,h_1)\colon f\to g\big)&=
\big(F(h_0), F(h_1)\big)\colon F(f)\longrightarrow F(g).
\end{align*}
Finally, given a natural 
transformation $\omega\colon F\Rightarrow G$, the transformation 
$\Ar  (\omega)\colon \Ar  (F) \Rightarrow \Ar  (G)$ has components 
\begin{equation}
    \label{Na koncert jede cely sbor vlakem}
\Ar (\omega)_{f}=(\omega_a,\omega_b)\colon F(f)\longrightarrow G(f).
\end{equation}

The comonad structure of $\Ar $ is given by precompositions with $m$
and $e$, using the natural isomorphisms
\[
[\dvedve\x\dvedve\to\C] \cong \big[\dvedve\to[\dvedve\to\C]\big] = \Ar
^2\C, \hbox { and }\
[\jednajedna \to \C]\cong \C. \]
Let us describe the comonadic comultiplication 
$\delta_\C := m^*_\C\colon \Ar \C\to \Ar ^2\C$ explicitly.
Firstly, 
\[
\raisebox{-1.8em}{$\delta_\C (f\colon a\to b)\ =$ \ }
\xymatrix{a \ar[r]^f \ar[d]_f  & b \ar[d]^\id
\\
b  \ar[r]^\id & b
}
\]
and, for a morphism $S = (h_0,h_1)\colon f \to g$ in $\Ar(\C)$ 
as in \eqref{equation:square}, we have
\[
\raisebox{-4.8em}{$\delta_\C (S) \ =$ \ }
\xymatrix@C=2.3em@R=2em{
& b  \ar[dd]|\hole^(.75){\id} \ar[rr]^{h_1}  &&d  \ar[dd]^{\id}
\\
a  \ar[ur]^{f}\ar[dd]_{f} \ar[rr]^(.7){h_0} &&c
\ar[dd]^(.2){g} 
\ar[ur]^g
\\
&b  \ar[rr]|\hole^(.3){h_1}&&d
\\
b \ar[ur]^\id \ar[rr]^(.4){h_1}&&
d \ar[ur]^\id
}
\]
Our convention is that the morphisms in $\
\Ar ^2(\C)$ always go from the left face of the cube to the right
face. 
The counit $\epsilon_\C:= e^*_\C \colon \C \to \Ar (\C)$ 
sends $f \colon a\to b \in \Ar (\C)$ to $a$, and a
morphism $S$ as in~\eqref{equation:square} to $h_0$.

By definition, an $\Ar $-coalgebra is a category $\C$ with a 
functor $\beta\colon \C\to \Ar (\C)$ that satisfies
\begin{subequations} 
\begin{equation}
\label{equation:coalgebra_counit}
 \epsilon_\C \circ \beta=\uu   
\end{equation} 
and 
\begin{equation}\label{equation:coalgebra_comultiplication}
\Ar (\beta)\circ \beta=\delta_\C\circ \beta.
\end{equation}
\end{subequations}
In elementary terms, the functor $\beta$ equips $\C$ with functorial
choices of commutative squares
\[
\raisebox{-2em}{$\beta_f = \hskip 1em $}
\xymatrix@C=3em{
\beta^0_a \ar[d]_{\beta^0_f} \ar[r]^{\beta_a}   & \beta^1_a  \ar[d]^{\beta^1_f} 
\\
\beta^0_b  \ar[r]^{\beta_b}  & \beta^1_b
}
\]
which are
specified for any $f\colon a\to b$ in $\C$.  Equation
\eqref{equation:coalgebra_counit} gives $\beta_f^0=f$. If we denote
$\beta_f^1 :\beta_a^1 \to \beta_b^1$ by
$u_f: u_a\to u_b$, the square $\beta_f$ assumes the
form
\[
\raisebox{-2em}{$\beta_f = \hskip 1em $}
\xymatrix@C=3em{
a \ar[d]_{f} \ar[r]^{\beta_a}   & u_a  \ar[d]^{u_f} 
\\
b  \ar[r]^{\beta_b}  &\ u_b .
}
\]
Equation \eqref{equation:coalgebra_comultiplication} evaluated on
a morphism $f\colon a\to b \in \C$ gives the equality of cubes
\[
\xymatrix@C=2.3em@R=2em{
&u_a \ar[dd]|\hole_(.75){\beta_{u_a}} \ar[rr]^{u_f}  &&u_b  \ar[dd]^{\beta_{u_b}}
\\
a  \ar[ur]^{\beta_a}\ar[dd]_{\beta_a} \ar[rr]^(.7){f} &&b
\ar[dd]^(.2){\beta_b} 
\ar[ur]^(.4){\beta_b}
\\
&u_{u_a}  \ar[rr]|\hole^(.3){u_{u_f}}&&u_{u_b}
\\
u_a \ar[ur]^(.55){u_{\beta_a}} \ar[rr]^(.55){u_f}&&
u_b \ar[ur]^{u_{\beta_b}}
} 
\hskip 2em
\raisebox{-4.5em}{$=$} \hskip 2em
\xymatrix@C=2.3em@R=2em{
&u_a \ar[dd]|\hole_(.75){\id} \ar[rr]^{u_f}  &&u_b  \ar[dd]^{\id}
\\
a  \ar[ur]^{\beta_a}\ar[dd]_{\beta_a} \ar[rr]^(.7){f} &&b
\ar[dd]^(.2){\beta_b} 
\ar[ur]^(.4){\beta_b}
\\
&u_a  \ar[rr]|\hole^(.3){u_f}&&u_b
\\
u_a \ar[ur]^{\id} \ar[rr]^(.55){u_f}&&
u_b \ar[ur]^{\id}
}
\]
which implies that $u_{\beta_a} = \beta_{u_a} = \id_{u_a}$, and thus
$u_{u_a} = u_a$ for an arbitrary object $a$ of $\C$, and
that $u_{u_f}={u_f}$  for any morphism $f$ of $\C$.

The above equations say that an $\Ar $-coalgebra is an idempotent
functor 
$u\colon \C \to \C$
together with a natural transformation
$\beta\colon \id_\C \Rightarrow u$, such that the equality of
transformations $u\beta=\beta u=\uu_u$ holds. This means that
$(u,\uu\colon u^2=u,\beta\colon \uu \Rightarrow u)$ is a monad on
$\C$.
\begin{example}
\label{example:pointed}
Any category with a terminal object $\texttt{t}$ is an $\Ar $-coalgebra
with $\beta_a:=!_a$, the unique map from $a$ to
$\texttt{t}$. Different choices of a (necessarily isomorphic)
terminal object give different, but isomorphic, coalgebra structures.
\end{example}

\section{From coalgebras to algebras}
\label{Dominik prileti v pondeli.}

Any comonad $(\Com,\delta,\ee)$ on $\Cat$ induces
a monad $\overline{\Com}$ on the category $\Cat_\Com$ of $\Com$-coalgebras,
which is the monad generated by the forgetful-cofree
adjunction
\[
\rule{0em}{1.5em}
\xymatrix@1@C=1em{     *{ \quad \ \  \quad \Cat_\Com\ } 
\ar@(dl,ul)[]^{\overline \Com}\ \ar@/^1em/[rr]&\bot \hskip .8em & *{\
\raisebox{.1em}{\Cat}  \raisebox{-.1em}{\rule{0em}{1.2em}}}
\ar@/^1em/[ll]}\ .
\]
Given a $\Com$-coalgebra $(\C,\beta_\C)$, the value of the induced monad
$\overline{\Com}$ on $\C$ is the cofree coalgebra $(\Com(\C),c_\C)$ with the
commultiplication $c_\C : = \delta_\C : \Com(\C) \to \Com^2(\C)$. The
components of the multiplication $\mu_\C$ and the unit 
$\eta_\C$ of the monad $\overline{\Com}$ are given by
$\mu_\C=\Com\ee_\C$ and $\eta_\C=\beta_\C$, $\C \in \Cat_\Com$. 
Indeed, $\mu_\C : \Com^2(\C) \to \Com(\C)$ and 
$\beta_\C\colon(\C,\beta_\C) \to (\Com(\C),c_\C)$ are morphisms 
of coalgebras,
and the unitality and associativity of $(\overline{\Com},\mu,\eta)$
follows from the properties of the comonad $\Com$ and coalgebra
structure maps $\beta_\C$.

We will apply this construction to the arrow comonad $\Ar$ introduced in
Section~\ref{Bures vladne vsem.}. 
The induced monad on the category of 
$\Ar$-coalgebras will be denoted by~$\overline{\Ar}$. It is, in fact, a
2-monad in the $2$-category $\Cat$, 
hence we can consider its normalized lax algebras which we
recall from \cite[Chapter~6]{Bourke}. `Lax normalized' means that the algebra
structure map $\K\colon \Ar(\C) \to \C$ interacts with the monad
multiplication via a
coherent natural transformation (the laxness) and the composite $\K\circ
\beta_\C\colon\C\to\C$ is strictly the identity on $\C$ (the
normalization). Here is the precise

\begin{definition}
\label{definition:lax_algebra}
Let $(\T,\mu,\eta)$ be a 2-monad on $\Cat$. A normalized lax
$\T$-algebra is a category~$\C$ together with a functor
$a\colon \T(\C)\to\C$ and a natural transformation $\phi\colon a \circ\T a
\Longrightarrow a\circ \mu_\C$ in
\begin{equation}
\label{V Mercine bude pekna zima.}
\xymatrix@R=2.5em@C=2.5em{
\T^2(\C) \ar[r]^{\T a} \ar[d]_{\mu_\C} 
& \T(\C) \ar[d]^a \ar@{=>}[dl]|{\phi\rule{.1em}{0pt}}
\\
\T(\C) \ar[r]^a & \C
}
\end{equation}
satisfying the following four conditions:

\begin{itemize}
\item[(i)]
$a \circ \eta_\C=\uu_\C$,
\item[(ii)] 
there is an equality of 2-cells
\[
\xymatrix@R=1.2em@C=1.7em{
\T^3(\C) \ar[rr]^{\T^2 a}\ar[dd]_{\mu_{\T(\C)}} && \T^2(\C) \ar[rd]^{\T a} \ar[dd]_{\mu_\C}
\\
&&& \T(\C) \ar[dd]^a \ar@{=>}[ld]|{\phi\rule{.1em}{0pt}}
\\
\T^2(\C)  \ar[rr]^{\T a} \ar[rd]_{\mu_\C} && \T(\C)
\ar@{=>}[ld]|{\phi\rule{.1em}{0pt}}
\ar[dr]^a
\\
&\T(\C) \ar[rr]^a && \C
}
\hskip 1em
\raisebox{-3em}{$=$} \hskip 1em
\xymatrix@R=1.2em@C=1.7em{
\T^3(\C) \ar[rr]^{\T^2 a}\ar[dd]_{\mu_{\T(\C)}} \ar[rd]^(.6){\T\mu_\C}   && 
\T^2(\C) \ar[rd]^{\T a} \ar@{=>}[ld]|(.45){\T{\phi\rule{.1em}{0pt}}}
\\
&\T^2(\C)  \ar[rr]^{\T a}   \ar[dd]_{\mu_\C}&& \T(\C) \ar[dd]^a 
\ar@{=>}[lldd]|{\phi\rule{.1em}{0pt}}
\\
\T^2(\C) \ar[rd]_{\mu_\C} &&
\\
&\T(\C) \ar[rr]^a && \C,
}
\]
\item[(iii)]
$\phi \cdot \T\eta_\C=\uu_a$, and
\item[(iv)]
$\phi \cdot \eta_{\T(\C)}=\uu_a$.
\end{itemize}
\end{definition}

Let again $(\Com,\delta,\ee)$ be a $2$-comonad 
and $\overline{\Com}=(\Com,\Com \ee,\beta)$ the
induced 2-monad on the category of $\Com$-coalgebras explicitly
described at the beginning of this section. 
Fix a $\Com$-coalgebra $\C$ with the structure map $\beta_\C : \C \to
\Com(\C)$ and recall that the structure map of the free
$\Com$-coalgebra on~$\C$ is $\delta_\C: \Com(\C) \to \Com^2(\C)$. 
We are going to describe, in
Theorem~\ref{theorem:big} below, a
useful way leading to a normalized lax algebra structures on $\C$.

Let $K\colon \Com(\C) \to \C$ be a morphism of $\Com$-coalgebras, i.e.\ a functor
satisfying  $\Com K\circ \delta_\C=\beta_\C\circ K$. Suppose also 
that $K\circ
\beta_\C = \uu_\C$ and assume the existence of a natural
transformation \hbox{$\nu\colon \uu_{\Com(\C)}\Rightarrow \beta_\C \circ
  \ee_\C$}. We will consider composed natural transformations 
\begin{equation}
\label{equation:rho_and_phi}
\kappa:=(K\cdot\nu)\colon
  K\Longrightarrow \ee_\C\ \text{ and }\ \phi:=(K\cdot \Com\kappa)\colon K\circ \Com K
  \Longrightarrow K\circ \Com \ee_\C.
\end{equation}

\begin{lemma}
\label{lemma:small_coherence}
We have
\begin{equation}
\label{equation:coherence_of_nu}
(\kappa\cdot \ee_{\Com(\C)})\square(\kappa\cdot \Com K)=(\kappa\cdot \Com \ee_\C)
\square (K\cdot \Com\kappa).
\end{equation}
In the above equation, as before, $\cdot$ denotes horizontal
composition and\hphantom{k}$\square$ \ 
vertical composition of transformations and functors.
\end{lemma}

\begin{proof}
The naturality of a transformation
$\alpha\colon F\Rightarrow G$ means that, for each 
morphism $f\colon x\to y$, the diagram
\begin{subequations}
\begin{equation}
\label{Za tyden jedu na Vanoce do Prahy.}
\xymatrix{F(x)\ar[d]_{F(f)}  \ar[r]^{\alpha_x}   & G(x)\ar[d]^{G(f)}
\\
F(y)  \ar[r]^{\alpha_y}&G(y)
}
\end{equation}
commutes.
Evaluating equation \eqref{equation:coherence_of_nu} 
on an object $z \in \Com^2(\C)$ leads to the diagram
\begin{equation}
\label{Vcera jsem byl na koncerte s Matyasem Kreckem.}
\xymatrix@C=4.5em{
K\Com K(z)  \ar[r]^(.37){(\kappa \cdot \Com K)(z)} \ar[d]_{(K \cdot \Com\kappa)(z)}  
& \ee_\C \Com K(z) =K\ee_{\Com(\C)}(z)
\ar[d]^{(\kappa\cdot \ee_{\Com(\C)})(z) =(\ee_\C\cdot \Com\kappa)(z)}
\\
K\Com \ee_\C(z) 
\ar[r]^(.37){(\kappa \cdot \Com \ee_\C)(z)}   &   
\ee_\C \Com \ee_\C(z) =   \ee_\C \ee_{\Com(\C)}(z)
}
\end{equation}
\end{subequations}
where the equalities follow from the $2$-naturality of the counit
$\ee\colon \Com \to \id$.  Taking~(\ref{Za tyden jedu na Vanoce do Prahy.})
with $F=K$, $G=\ee_\C$, $x=\Com K(z)$, $y=\Com \ee_\C(z)$, $f=\Com\kappa(z)$
and $\alpha=\kappa$,
we see that~(\ref{Vcera jsem byl na koncerte s Matyasem Kreckem.}) 
expresses the naturality of~$\kappa$. Since the object $z$ was arbitrary,
equality~\eqref{equation:coherence_of_nu} holds.
\end{proof}

\begin{theorem}
\label{theorem:big}
Let $(\Com,\delta,\ee)$ be a 2-comonad and
$\overline{\Com}=(\Com,\Com \ee,\beta)$ the induced 2-monad on the
category of \/ $\Com$-coalgebras.  Let $\beta_\C\colon \C\to \Com(\C)$ be
a coalgebra structure on a category $\C$ and
$\nu\colon \uu_{\Com(\C)}\Rightarrow \beta_\C \circ \ee_\C$ a natural
transformation such that $\ee_\C\cdot\nu=\uu_{\ee_\C}$, and
$\nu\cdot\beta_\C=\uu_{\beta_\C}$.  Finally, let
$K\colon \Com(\C) \to \C$ be a morphism of \/ $\Com$-coalgebras such
that $K\circ\beta_\C=\uu_\C$, and $\phi$ be as in
\eqref{equation:rho_and_phi}.  If $\phi\cdot \delta_\C = \uu_K$, then $(\C,K,\phi)$ is a normalized lax $\overline{\Com}$-algebra. Moreover, the coherence $\phi$ in any normalized lax $\overline{\Com}$-algebra $(\C,K,\phi)$ is uniquely determined by $K$ and $\nu$.
\end{theorem}

\begin{proof}
We need to verify the four conditions formulated in
Definition~\ref{definition:lax_algebra}. Condition (i) is assumed 
by the theorem. The equality of the
composites of natural transformations required in (ii) follows from the
chain of equalities
\begin{subequations}
\begin{align}
 \label{A}    (\phi\cdot \Com \ee_{\Com (\C)}) \square (\phi\cdot \Com ^2K)=&
(K\cdot \Com \kappa\cdot \Com \ee_{\Com (\C)}) \square (K\cdot \Com \kappa\cdot \Com ^2K)
\\
\label{B}  =& K\cdot \Com   \big((\kappa\cdot \ee_{\Com (\C)}) \square ( \kappa\cdot
                   \Com K)\big)
\\
\label{C}  =&(K\cdot \Com \big((\kappa\cdot \Com \ee_\C)\square (K\cdot
                  \Com \kappa)\big)
\\
\label{D}  =&(K\cdot \Com \kappa \cdot \Com ^2\ee_\C)\square 
(K\cdot \Com K\cdot \Com ^2\kappa )= 
(\phi\cdot \Com ^2\ee_\C)\square(K\cdot \Com \phi).
\end{align}
\end{subequations}
The terms in~\eqref{A} represent the left hand side of the equation
in~(ii),~\eqref{B} is obtained by factoring out $K\cdot \Com$,~\eqref{C} follows
from~\eqref{equation:coherence_of_nu}, and~\eqref{D} is the right hand
side of the equation in~(ii). Condition~(iii) follows from the
straightforward chain of equations
\[
\phi\cdot \Com {\beta_\C}=K\cdot \Com \kappa\cdot \Com {\beta_\C}=K\cdot (\Com K\cdot \Com \nu)\cdot \Com {\beta_\C}=(K\circ \Com  K)\cdot \Com \nu\cdot \Com {\beta_\C}=(K\circ \Com  K)\cdot\uu_{\Com {\beta_\C}}=\uu_K.
\]
Condition~(iv) leading to   $\phi\cdot \beta_{\Com(\C)} = \phi\cdot
\delta_\C = \uu_K$ is satisfied by the remaining
assumption of the theorem.

To show the uniqueness of coherence, let $(\C,K,\phi)$ be a normalized
lax $\overline{\Com}$-algebra. In particular, we assume that
$\phi\circ\id_{\Com(\beta_\C)}=\phi\cdot\Com(\beta_\C)=\id_K$ and we remind
that $\id_{\ee_\C}\circ\nu=\ee_\C\cdot \nu=\uu_{\ee_\C}$ is assumed by
the theorem.  The uniqueness of $\phi$ follows from the uniqueness of the
horizontal composition of 2-cells in the diagram
\[
\xymatrix@R=1em{
& \ar@{=>}[dd]^(.4){\Com\nu} && \Com(\C) \ar@{=>}[dd]^\phi \ar@/^1em/[rd]^K 
\\
\Com^2(\C)   \ar@/_1em/[rd]^{\Com\epsilon_\C}
\ar@/^2.9em/[rr]^{\id}  
&& \Com^2(\C) 
\ar@/^1em/[ru]^{\Com K} \ar@/_1em/[rd]^{\Com\epsilon_\C} 
&& \hphantom{m}\C \hphantom{m}
\\
&\Com(\C)  \ar@/_1em/[ru]^{\Com\beta_\C}  && \Com(\C)  \ar@/_1em/[ru]^K
}
\]
which means the commutativity of the diagram
\[
\xymatrix@C=6em@R=3em{K \circ \Com K  
\ar@{=>}[r]^(.37){(K \circ \Com K ) \cdot \Com \nu }
\ar@{=>}[d]_\phi
& (K \circ \Com K )\circ (\Com \beta_\C 
  \circ\Com \epsilon_\C )   \ar@{=>}[d]^{\phi \cdot (\Com \beta_\C 
    \circ \Com \epsilon_\C) }
\\
K \circ \Com \epsilon_\C  \ar@{=>}[r]^{(K \circ \Com \epsilon_\C )
  \cdot \Com(\nu)}
& K \circ \Com \epsilon_\C 
}
\]
of natural transformations. In that diagram the vertical right and
horizontal bottom transformations are identities by assumption.
\end{proof}

\begin{remark}\label{remark:adjunction}
The conditions $\ee_\C\cdot\nu=\uu_{\ee_\C}$ and
$\nu\cdot\beta_\C=\uu_{\beta_\C}$ in Theorem~\ref{theorem:big}
are equivalent to the adjunction
$(\uu,\nu)\colon \ee_\C\dashv \beta_\C$ with
unit $\nu\colon \id_{\Com(\C)}\Rightarrow \beta_\C \circ \ee_\C$ and
counit $\uu\colon \ee_\C\circ\beta_\C=\uu_{\C}$.  If we define
$\theta\colon \beta_\C\circ K \Rightarrow\uu$~as
\[
\theta:=\Com K\cdot \Com\nu\cdot\delta_\C
\] 
we get another adjunction
$(\theta,\uu)\colon \beta_\C\dashv K\colon \Com(\C)\to \C$, and
Theorem~\ref{theorem:big} then follows from the dual of
\cite[Lemma~2.4.16]{Milos}.
\end{remark}

A pointed category is a category with a chosen object $0$ which
is both initial and terminal. For an object $a$, we denote the unique
maps by $!_a\colon a\to0$ and~$!^a\colon 0\to a$.
The zero morphism $0\colon a \to b$ is, 
by definition, a~morphism that factorizes through $0$.

We will now apply Theorem~\ref{theorem:big} to the case when $\Com=\Ar$ is the
arrow category functor with the
comonad structure $(\Ar ,\delta,\ee)$ described in 
Section~\ref{Bures vladne vsem.}. As explained in
Example~\ref{example:pointed}, each  pointed category $\C$ 
is an $\Ar $-coalgebra with $\beta_\C(a)= \ !_a$
 and  $\beta_\C(f\colon a\to b)=(!_a,!_b)\colon f\to 0$. In this
 particular case we have a canonical explicit natural transformation 
$\nu\colon \uu_{\Ar (\C)}\Rightarrow
    \beta_\C \circ \ee_\C$ given by 
\begin{equation}
\label{Pul pate a uz je tma.}
\nu_{(f\colon a\to b)}=(\uu_a,!_b)\colon f\longrightarrow \ !_a.
\end{equation}

\begin{proposition}
\label{Ke sboru pristoupim az v Bilovicich.}
Let \ $\C$ be a pointed category,  $\K \colon \Ar (\C) \to \C$ a functor and
$\phi \colon \K  \circ \Ar \K  \Rightarrow \K  \circ \Ar \ee_\C$ the transformation
defined via $\nu$ in~(\ref{Pul pate a uz je tma.}) and \/ $\K $
by the formula in~\eqref{equation:rho_and_phi}. 
If \/ $\K \circ\beta_\C=\uu_\C$, $\Ar  \K \circ
\delta_\C=\beta_\C\circ \K $ and $\phi\cdot \delta_\C = \uu_\K $, 
then $(\C,\K ,\phi)$ is a normalized lax
$\overline{\Ar }$-algebra. Moreover, the coherence $\phi$ in any normalized lax
$\overline{\Ar }$-algebra $(\C,\K,\phi)$  is
uniquely determined by $\K $.
\end{proposition}

\begin{proof}
The conditions $\ee_\C\cdot\nu=\uu_{\ee_\C}$, and
$\nu\cdot\beta_\C=\uu_{\beta_\C}$ 
are obvious, so Theorem~\ref{theorem:big} applies.
\end{proof}

\section{Kernels as algebras}
\label{Zacina Zimni skola. Idealni snehove podminky ale ja trcim v Bonnu!}

From this section on, 
we will always assume that the category $\C$ 
is pointed and understood as an $\Ar $-coalgebra with $\beta_\C(a)= \ !_a$
and  $\beta_\C(f\colon a\to b)=(!_a,!_b)\colon f\to 0$, cf.\
Example~\ref{example:pointed}. 
\begin{definition}
\label{definition:algebraic_kernels}
A pointed category $\C$ has  algebraic kernels if for
every morphism $f\colon a\to b$ in $\C$ there is a choice of a  {\/\em
  kernel} of $f$ which is a
morphism $\kappa_f\colon \ker f \to a$ 
such that
\begin{itemize}
\item[(i)] 
the composite \, \redukce{$\ker f \stackrel {\kappa_f}\longrightarrow a \stackrel
  f\longrightarrow b$} is the zero morphism, 
\item[(ii)]  
$\kappa_{\uu_a}=!^a$
and $\kappa_{!_a}=\uu_a$ for any~$a\in \C$ (the
normalization), and
\item[(iii)] the following universal property holds: for any $g \colon c
  \to a$ such that $fg$ is the zero morphism, there exists a unique $h
  \colon c \to \ker f$ in the diagram
\[
\xymatrix@R=.5em{
\ker f \ar[rd]^{\kappa_f}
\\
&a \ar[r]^f &b
\\
c \ar[ur]^g \ar@{-->}@/^1em/[uu]^h
}
\]
with $\kappa_fh=g$.
\end{itemize}
\end{definition}
Assuming the axiom of choice, every category with
kernels has algebraic kernels. The notions of kernels
and algebraic kernels are then equivalent, and we shall always 
assume kernels to be algebraic without further mention. The main result of this
section reads

\begin{theorem}
\label{theorem:kernels}
Let $\C$ be a pointed category. The~following conditions are equivalent.
    \begin{itemize}
        \item[(i)] The category $\C$ carries a normalized lax $\overline{\Ar}$-algebra structure.
     \item[(ii)] The category $\C$ has kernels.
\end{itemize}
\end{theorem}

Before proving the theorem we describe the structure of a normalized lax
$\overline{\Ar}$-algebra explicitly. 
The data of Definition~\ref{definition:lax_algebra}
for the monad
$\overline{\Ar}$ translate to $\T=\Ar$, $\mu_\C=\Ar\ee_\C=\Ar e^*_\C$, $\eta_\C=\beta_\C$ or 
simply~$\beta$ when~$\C$ is understood, and $\eta_{\T(\C)}=\beta_{\Ar(\C)}=\delta_\C=m^*_\C$,
the structure map of the cofree coalgebra~$\Ar(\C)$. Further, we will
write $\K$ instead of
$a$. A normalized lax $\overline{\Ar}$-algebra is then 
an $\Ar$-coalgebra $(\C,\beta\colon \C \to \Ar(\C))$ with an algebra
structure map $\K\colon \Ar(\C) \to \C$ and a coherent natural
transformation $\phi$ in
\begin{equation}
\label{Uz je druhy den kolem nuly.}
\xymatrix@R=2.5em@C=2.5em{
\Ar^2(\C) \ar[r]^{\Ar \K} \ar[d]_{\Ar\ee_\C} 
& \Ar(\C) \ar[d]^\K \ar@{=>}[dl]|{\phi\rule{.1em}{0pt}}
\\
\Ar(\C) \ar[r]^\K &\ \C \ .
}
\end{equation}
We will 
use the notation \hbox{$\K\! f \ \triangleright \ a \Xrightarrow{f} b$}
inspired by \cite{duo} to express that $\K f$ is the value of~$\K$ on a morphism $f : a \to b$. 
A typical diagram capturing the
lax $\overline{\Ar}$-algebra structure~is
\begin{equation}
\label{Komplikace s nahranim clanku se vrsi.}
\xymatrix@C=-.2em@R=.0em{\K({\K} S) \ar@{->}[dd]_{\phi_S} 
& \triangleright & \K \ar[rrrrrr]^{\K S} f &&&&&& \K g
\\
&& \raisebox{.8em}{\rotatebox{-90}{$\triangleright$}} &&&&&& 
\raisebox{.8em}{\rotatebox{-90}{$\triangleright$}}
\\
\K h_0 &  \triangleright  & a \ar[dddd]_f  \ar[rrrrrr]^{h_0} &&&&&&
c  \ar[dddd]^g
\\
\\
&&&&&\raisebox{.2em}{$S$}
\\
\\
&& b  \ar[rrrrrr]^{h_1} &&&&&& d\, . \hspace {-.8em}
}
\end{equation}
The normalizing condition $\K\circ \beta=\uu_\C$ gives, for $a\in
\C$, $\K\beta_a = a$. For a morphism $f$ in $\C$ and the
associated morphism $\beta_f=(f,0)\colon !_a \to !_b$ in $\Ar(\C)$ it
gives the equality
$\K \beta_f = f$. Together with the normalizing condition $\phi\cdot \Ar\beta=\id_\K$, this is expressed~by 
\begin{equation}
\label{diagram: first unit coherence of A}
\xymatrix@C=-.2em@R=.0em{\K f \ar@{->}[dd]_{\phi_{\beta_f}=\id_{\K f}} 
& \triangleright & a \ar[rrrrrr]^{f} &&&&&& b
\\
&& \raisebox{.8em}{\rotatebox{-90}{$\triangleright$}} &&&&&& 
\raisebox{.8em}{\rotatebox{-90}{$\triangleright$}}
\\
\K f &  \triangleright  & a \ar[dddd]_{!_a}  \ar[rrrrrr]^{f} &&&&&&
b  \ar[dddd]^{!_b}
\\
\\
&&&&&\raisebox{.2em}{$\beta_f$}
\\
\\
&& 0  \ar[rrrrrr]^{0} &&&&&& 0\, . \hspace {-.8em}
}
\end{equation}
The functor $\K \colon \Ar(\C) \to \C$ must be an $\Ar$-coalgebra homomorphism from
the cofree $\Ar$-coalgebra $(\Ar(\C),\delta_\C)$ to $(\C,\beta)$. This means
that the diagram
\begin{equation}
\label{Dnes jsem vlastne v Bonnu druhy tyden.}
\xymatrix{\Ar(\C) \ar[r]^\K   \ar[d]_{\delta_\C} & \C \ar[d]^\beta
\\
\Ar^2(\C)\ar[r]^{\Ar\K} & \Ar(\C)
}
\end{equation}
commutes. Evaluating~(\ref{Dnes jsem vlastne v Bonnu druhy tyden.}) at 
$f\colon a\to b \in \C$ gives the equality of morphisms
\begin{equation}
\label{Dnes zacina seminar v 16:30 - magori.}
\big(\Ar\K( \delta_\C(f)) \colon \K f \longrightarrow \K \id_b\big)\  
=\ (\beta_{\K f} \colon \K f \longrightarrow   0).
\end{equation}
Together with the second normalizing condition $\phi\cdot \delta_\C = \id_{\K}$, this is expressed via the diagram
\begin{equation}
\label{diagram: second unit coherence of A}
\xymatrix@C=-.2em@R=.0em{\K f \ar@{->}[dd]_{\phi_{\delta_\C(f)}=\id_{\K f}} 
& \triangleright & \K f \ar[rrrrrr]^{\beta_{\K f}} &&&&&& 0
\\
&& \raisebox{.8em}{\rotatebox{-90}{$\triangleright$}} &&&&&& 
\raisebox{.8em}{\rotatebox{-90}{$\triangleright$}}
\\
\K f &  \triangleright  & a \ar[dddd]_{f}  \ar[rrrrrr]^{f} &&&&&&
b  \ar[dddd]^{\id_b}
\\
\\
&&&&&\raisebox{.2em}{$\delta_\C(f)$}
\\
\\
&& b  \ar[rrrrrr]^{\id_b} &&&&&& b\, . \hspace {-.8em}
}
\end{equation}
In particular, 
\begin{equation}
\label{Za hodinu se musim zacit balit na koncert.}
    \K \uu_b=0 
\end{equation}
for any $f$ with target $b$.

Evaluating~(\ref{Dnes jsem vlastne v Bonnu druhy tyden.}) at the
morphism $S = (h_0,h_1)$ represented by the square in~(\ref{equation:square}) leads to the equality of diagrams
\[
\xymatrix{
\K f  \ar[r]^{\K S} \ar[d]_{\Ar \K (\delta_\C(f))}   & \K g \ar[d]^{\Ar \K (\delta_\C(g))}
\\
\K\id_b \ar[r]^{\K P} & \K \id_d 
}
\ \raisebox{-1.7em}{$=$} \
\xymatrix{
\K f  \ar[r]^{\K S} \ar[d]_{\beta_{\K f}}   & \K g \ar[d]^{\beta_{\K g}}
\\
0 \ar[r]^0 & 0
}
\]
in which $P = (h_1,h_1)\colon \id_b \to \id_d$.
Combined with~(\ref{Dnes zacina seminar v 16:30 - magori.}) it gives
the equality $\K P  = 0$ 
which can be expressed as
\[
\xymatrix@C=-.1em@R=-.2em{
0 \ar[rrrrrr]^{0} &&&&&& 0
\\
\raisebox{.8em}{\rotatebox{-90}{$\triangleright$}}  &&&&&& 
\raisebox{.8em}{\rotatebox{-90}{$\triangleright$}}
\\
b \ar[dddddd]_{\id_b} \ar[rrrrrr]^{h_1} &&&&&& d  \ar[dddddd]^{\id_d} 
\\
\\
\\ 
&&& P
\\
\\
\\
b  \ar[rrrrrr]^{h_1} &&&&&& d \tecka
}
\]

\begin{proof}[Proof of Theorem~\ref{theorem:kernels}]
Let
$\C=(\C,\beta_\C,\K,\phi\big)$
  be a lax normalized $\overline{\Ar}$-algebra 
and recall the natural transformation 
$\nu\colon \uu_{\Ar(\C)}\Rightarrow
    \beta_\C \circ \ee_\C$ of \eqref{Pul pate a uz je tma.} given by $$\nu_{(f\colon a\to b)}=(\uu_a,!_b)\colon f\longrightarrow !_a.$$ For any $f\colon a \to b$ of $\C$ consider the component 
\begin{equation}\label{equation:AKvf}
\kappa_f :=(\K\cdot\nu)_f=\K\left(
\raisebox{2em}{\xymatrix{
	a \ar[r]^{\id_a}\ar[d]_f& a \ar[d]^{!_a}
\\
	b \ar[r]^{!_b}& 0
}}
\right)=\left(
	\K(f) \xrightarrow{\K(\uu_a,!_b)}\K(!_a)
   \right) =
	\left(\K(f) \xrightarrow{\K(\uu_a,!_b)} a
\right).
\end{equation}
The object $\K (f)$ is therefore connected to the domain of $f$ by
a component of the natural transformation $\kappa =\K\cdot \nu$. 
Thus for any morphism
$S=(h_0,h_1)\colon f\to g$ of $\Ar(\C)$   
the triangles \raisebox{.6em}{\rotatebox{-90}{$\triangleright$}} 
in diagram \eqref{Komplikace s nahranim clanku se vrsi.} can 
be replaced by actual maps $\kappa_f$ and $\kappa_g$ as in
\begin{equation}
\label{Komplikace s nahranim clanku se vrsi. A stale dal vrsi.}
\xymatrix@R=1.2em{
 \K \ar[r]^{\K S}\ar[d]_{\kappa_f} f & \K g\ar[d]^{\kappa_g}
\\
 a \ar[dd]_f  \ar[r]^{h_0} &
c  \ar[dd]^g
\\
\\
 b  \ar[r]^{h_1} &d
}
\end{equation}
such that the top square commutes. In fact, the top square
$(\kappa_f,\kappa_g)\colon \K S\to h_0$ is the component of
$(\Ar\kappa )_S=(\kappa_f,\kappa_g)$, cf.~\eqref{Na koncert jede cely sbor
  vlakem}, of the natural transformation
$\Ar\kappa=\Ar\K\cdot \Ar\nu\colon \Ar\K \Rightarrow \Ar\ee_\C$. By \eqref{Za hodinu se
  musim zacit balit na koncert.}, $\K\uu_a=0$, hence
$\kappa_{\uu_a}=!^a\colon 0\to a$. Further,
$\nu_{!_a}=(\uu_a,!_0)=\uu_{!_a}$, so
$\kappa_{!_a}=\K\uu_{!_a}=\uu_{\K!_a}=\uu_a$ by functoriality of $\Ar\K$.
Moreover, $f\kappa_f=0$ for any morphism of $\C$, which follows from the
following instance of~\eqref{Komplikace s nahranim clanku se vrsi. A
  stale dal vrsi.}, where $S=\delta_\C(f)$.
\begin{equation}
\xymatrix@R=1.2em@C=3em{
 \K \ar[r]\ar[d]_{\kappa_f} f & 0 \ar[d]
\\
 a \ar[dd]_f  \ar[r]^{f} &
b  \ar[dd]^{\uu_b}
\\
\\
 b  \ar[r]^{\uu_b} &b\, . \hspace {-.8em}
}
\end{equation}
We will show that $\kappa_f: \K f \to a$ is a kernel of $f\colon a \to b$.

Let $g\colon c\to a$ be a morphism such that $fg=0$. We must prove the
existence of a unique $h : c \to \K f$ in the diagram
\begin{equation}
\label{Targa Florio}
\xymatrix@R=.5em{
\K f \ar[rd]^{\kappa_f}
\\
&a \ar[r]^f &b \tecka
\\
c \ar[ur]^g \ar@{-->}@/^1em/[uu]^h
}
\end{equation}
Let us denote the morphism $(g,!^b)\colon !_c\to f$ by $T$. 
The commutativity of the diagram 
\begin{equation}
\xymatrix@R=1.2em@C=3em{
 c \ar[r]^{\K T}\ar[d]_{\kappa_{!_c}=\uu_c} & \K f \ar[d]^{\kappa_f}
\\
 c \ar[dd]_{!_c}  \ar[r]^{g} &
a  \ar[dd]^{f}
\\
\ar@{}[r]|T &
\\
 0  \ar[r]^{!^b}&b\, . \hspace {-.8em}
}
\end{equation}
implies that
$\K T\colon c\to \K f$ satisfies $\kappa_f\circ\K T=g$. The morphism $\K
T\colon c\to \K f$ will therefore be our candidate for the dashed arrow
in~(\ref{Targa Florio}).  
It remains to
prove that $h:= \K T$ is the unique morphism making~(\ref{Targa
  Florio}) commutative.

Let $h\colon c\to \K (f)$ be any morphism with $\kappa_f \circ h=g$.
The uniqueness $h=\K(T)$ follows from the value $\Ar \K (X)$ on the
cube
\[
X=\big(\uu_{!_c} \ , \ \Ar\kappa_{\delta_\C(f)}\big)\colon \Ar\beta(h)\to
T
\] cf.~\eqref{equation:boring_cube}, viewed as a morphism in
$\Ar^2(\C)$. By Proposition~\ref{Ke sboru pristoupim az v Bilovicich.}, it
holds $\phi=(\K \circ \Ar\K) \cdot \Ar\nu$, so the normalizing condition (iv) of
Definition~\ref{definition:lax_algebra} gives
$$\id_{\K}=\phi\cdot \delta_\C = (\K\circ \Ar\K) \cdot \Ar\nu \cdot \delta_\C=\K\cdot \Ar(\K\cdot \nu)\cdot \delta_\C=\K\cdot \Ar\kappa\cdot{\delta_\C}.$$ Note that
\[
(\Ar\kappa \cdot \delta_\C)
(f)=\Ar\kappa_{\delta_\C(f)}=(\kappa_f,!^b)\colon
!_{\K f}\to f
\] 
is the back square of $X$. Hence the back edge of the square
$\Ar\K (X)$ on the top of Figure~\ref{equation:boring_cube} is
$\K\big(\Ar\kappa_{ \delta_\C(f)}\big)=\id_{\K(f)}$, which proves
$h=\K (T)$.
\begin{figure}
\hphantom{a}\xymatrix{
& 
 \K f
  \ar[rr]^{\K\big(\Ar\kappa_{\delta_\C(f)}\big)=\id_{\K f}} && \K f 
\\
\ar[ru]^(0.3){h} c
\ar[rr]^{\id_c} 
&& c \ar[ru]_{\K T}
\\
& 
\ar[dd]^(.3){}|\hole \K f
  \ar[rr]^{\kappa_f} && 
\ar[dd]^f a
\\
\ar[dd] \ar[ru]^{h} c
\ar[rr]^(0.7){\id_c} 
&& c \ar[ru]^g \ar[dd]
\\
& \ar[rr]^{}|\hole 0 && b
\\
 \ar[rr] 0 \ar[ru] && 0 \ar[ru]
}
\caption{A cube with halo.
\label{equation:boring_cube}}
\end{figure}
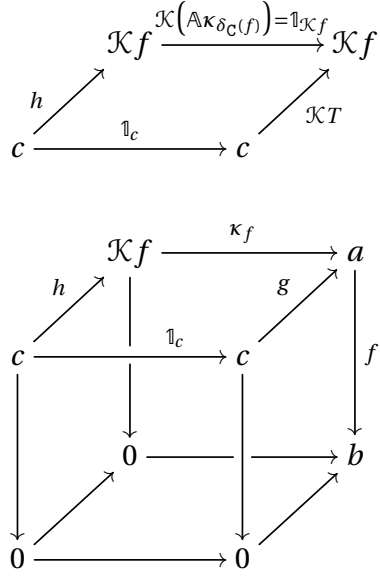

\begin{center}
* * * * *
\end{center}

Conversely, assume that $\C$ is equipped with kernels $\kappa_f :
\ker f \to a$. The
kernels are monomorphisms, which we emphasize by denoting them by
arrows with forked tails whenever it makes sense. We
define the action of the normalized lax $\overline{\Ar}$-algebra
structure on $f\in \Ar(\C)$ as 
\[
\K f:=\ker f.
\]  
Let $S=(h_0,h_1)\colon f\to g$ be the commutative
square representing a morphism of $\Ar(\C)$. 
Since $gh_0\kappa_f=h_1f\kappa_f=0$, there is a unique morphism
from $\ker f$ to $\ker g$ such that the upper square of the 
diagram 
\[
\xymatrix@R=.5em@C=.5em{\ker \ar@{-->}[rr] \ar@{>->}[dd]_{\kappa_f}  f 
&& \ker g \ar@{>->}[dd]^{\kappa_g} 
\\
\\
a \ar[rr]^{h_0}  \ar[dd]_f && c \ar[dd]^g
\\
&S
\\
b \ar[rr]^{h_1}&& d
}
\]
commutes.  We define $\K S: \ker f \to \ker g$ to be the dashed arrow in the above
diagram. The functoriality of this assignment 
follows from the universal property of kernels. Moreover, $\kappa_f$ assemble into a natural transformation $\kappa\colon \K \to \epsilon_\C.$  
The rest follows from Proposition~\ref{Ke sboru pristoupim az v
  Bilovicich.},
once we check the conditions $\K\circ\beta_\C=\uu_\C$, $\Ar \K\circ
\delta_\C=\beta_\C\circ \K$ and $\phi\cdot \delta_\C = \uu_\K$.

First, the condition $\K\circ\beta_\C=\uu_\C$ follows directly 
from normalizing assumption $\kappa_{!_a}=\uu_a$ for any~$a\in \C$. The condition $\Ar \K\circ
\delta_\C=\beta_\C\circ \K$ follows from the diagram 
\[
\xymatrix@R=.5em@C=1.5em{\ker f \ar@{-->}[rr]^(.55){\Ar\K(\delta_\C(f))=!_{\ker f}} \ar@{>->}[dd]_{\kappa_f}   
&& 0 \ar@{>->}[dd]^{\kappa_{\id_b}=!^b} 
\\
\\
a \ar[rr]^{f}  \ar[dd]_f && b \ar[dd]^{\id_b}
\\
&\delta_\C(f)
\\
b \ar[rr]^{\id_b}&& b
}
\]
where we used the second normalizing assumption $\kappa_{\id_b}=!^b$.

To show $\phi\cdot \delta_\C = \id_\K$ we first need to
describe the components of the natural transformation
$\phi=\K \cdot \Ar(\K \cdot \nu)$ as defined 
in \eqref{equation:rho_and_phi}. 
The component $(\K\cdot\nu)_f$ comes from the diagram 
\[
\xymatrix@R=.5em@C=.5em{\ker f \ar@{-->}[rr]^{\K\nu_f} \ar@{>->}[dd]_{\kappa_f} 
&& a \ar@{>->}[dd]^{\kappa_{!_a}=\id_a} 
\\
\\
a \ar[rr]^{\id_a}  \ar[dd]_f && a \ar[dd]^{!_a}
\\
&\nu_f
\\
b \ar[rr]^{!_b}&& 0
}
\]
hence $(\K\cdot \nu)_f=\kappa_f$. 
So $\phi_S=\big(\K \cdot \Ar(\K\cdot \nu)\big)_S=\K(\kappa_f,\kappa_g)$, 
which is depicted in the the diagram
\begin{equation}
\label{Jarka byla cely den doma jak bylo hrozne pocasi.}
\xymatrix@C=.5em@R=.5em{
\raisebox{-.4em}{} \ker \K S \ \ar@{>-->}[ddd]_{\phi_S}
\ar@{>->}[rr]^{\kappa_{\K S}}
&& \raisebox{-.4em}{} \ker f  \ar@{>->}[ddd]_{\kappa_f} \ar[rrrr]^{\K S}  &&&&
\raisebox{-.4em}{} \ker g  \ar@{>->}[ddd]^{\kappa_g}
\\
\\
&&&&&&
\\
\ker h_0 \ \ar@{>->}[rr]^{\kappa_{h_0}} && a \ar[dd]_f  \ar[rrrr]^{h_0} &&&&
c  \ar[dd]^{g}
\\
&&&&\raisebox{.2em}{$S$}
\\
&& b  \ar[rrrr]^{h_1} &&&& g
}
\end{equation}
Evaluating at $S=\delta_\C(f)$, the condition $\phi\cdot \delta_\C = \id_\K$ follows from the diagram
\[
\xymatrix@C=.5em@R=.5em{
\raisebox{-.4em}{} \ker f \ \ar@{>-->}[ddd]_{\phi_{\delta_\C (f)}}
\ar@{>->}[rr]^{\id_{\ker f}}
&& \raisebox{-.4em}{} \ker f  
\ar@{>->}[ddd]^{\kappa_f} \ar[rrrr]^{!_{\ker f}}  &&&& 
 \raisebox{-.4em}{} 0  \ar@{>->}[ddd]^{!^b}
\\
\\
&&&&&&
\\
\ker f \ \ar@{>->}[rr]^{\kappa_{f}} && a \ar[dd]_{f}  \ar[rrrr]^{f} &&&&
b  \ar[dd]^{\id_b}
\\
&&&&\raisebox{.2em}{$\delta_\C (f)$}
\\
&& b  \ar[rrrr]^{\id_b} &&&& b,
}
\]
in which the top left square commutes and the map $\kappa_f$ is a monomorphisms.
\end{proof}

 \begin{proposition}
\label{Je streda a uz melu z posledniho.}
 Any two normalized lax $\overline \Ar$-algebras on a pointed category
 $\C$ are connected by a pseudoisomorphism
 whose functor part is the identity functor on $\C$.
 \end{proposition}
 
\begin{proof}
By Theorem~\ref{theorem:kernels}, the two normalized lax algebra
structures in the proposition are the same as two systems $\K'$ and $\K''$
of kernels on $\C$.
The proposition asserts the existence
of a natural isomorphism $\Phi \colon \K' \Rightarrow \K''$ in the diagram
\[
\xymatrix@R=2.5em@C=2.5em{
\Ar(\C)  \ar[r]^{\K'} \ar@{=}[d]   & \C \ar@{=}[d] \ar@{<=>}[ld]|\Phi
\\
\Ar(\C)  \ar[r]^(.6){\K''}  & \C
}
\]
which satisfies the coherence
\begin{equation}
\label{Jarusku prijali mezi ochotniky.}
\xymatrix@R=2.5em@C=2.5em{
\Ar^2(\C)  \ar@{=}[d]  \ar[r]^{A\K'}  
& \ar@{<=>}[ld]|{\Ar\Phi}  \ar@{=}[d]  \ar[r]^{\K'} \Ar(\C)
& \ar@{=}[d] \C \ar@{<=>}[ld]|\Phi
\\
\Ar^2(\C) \ar@/_/[rd]_{\Ar\ee_\C}  \ar[r]^{A\K''}    
&   \ar@{=>}[d]|{\phi''}  \ar[r]^{\K''} \Ar(\C) & \C
\\
&\Ar(\C) \ar@/_/[ur]_{\K''}
}
\hskip 2em
\raisebox{-2.5em}{$=$} \hskip 1.8em
\xymatrix{
\Ar^2(\C)  \ar[r]^{\Ar \K'} \ar[d]_{\Ar \ee_\C}  & \Ar(\C) \ar[d]^{\K'} 
\ar@{=>}[ld]|{\phi'}
\\
\Ar(\C) \ar@{=}[d] \ar[r]^{\K'} & \C \ar@{=}[d] \ar@{<=>}[ld]|\Phi
\\
\Ar(\C) \ar[r]^(.6){\K''} & \C
}
\end{equation}
which involves the associated
transformations $\phi'$ and $\phi''$ in~(\ref{Jarka byla cely den
  doma jak bylo hrozne pocasi.}). Thanks to the normalization, the
standard coherence involving the monad units reduces to $\Phi_{!_a} =
\id_a$ for each $a\in \C$. 

The transformation $\Phi$ is constructed as follows. Given $f \colon a \to
b \in A\C$, then $\Phi_f\colon \K' f \to \K'' f$ 
is the unique morphism in the diagram
\[
\xymatrix@R=.5em{
\K'' f \ar@{>->}[rd]^{\kappa''_f}
\\
&a \ar[r]^f &b \ .
\\
\K' f \ar@{>->}[ur]_{\kappa'_f} \ar@{-->}@/^1em/[uu]^{\Phi_f}
}
\]
Its existence follows from the universal property of the kernels 
and it holds $\kappa''_f\circ\Phi_{f} =
\kappa'_f.$ 

The condition $\Phi_{!_a} = \id_a$ clearly holds.
To verify~(\ref{Jarusku prijali mezi ochotniky.}), it suffices to
check the equality 
\begin{equation}
\label{equation:morphism coherence}
\Phi_{h_0}\circ\phi'_S=\phi''_S\circ\K''\Phi_S\circ \Phi_{\K'S}.
\end{equation}
on the square $S =(h_0,h_1)\colon f\to g \in\Ar^2(\C)$.
We compute
\begin{subequations}
\begin{align}
\nonumber 
\kappa''_{h_0}\circ\Phi_{h_0}\circ\phi'_S &=
\kappa'_{h_0}\circ\phi'_S 
= \kappa'_f\circ \kappa'_{\K'S}
= \kappa'_f\circ \kappa''_{\K'S}\circ \Phi_{\K'S}
\\
\label{28d}&= \kappa''_f\circ \Phi_f\circ\kappa''_{\K'S}\circ \Phi_{\K'S}\\
\label{28e}&= \kappa''_f\circ \kappa''_{\K''S}\circ \K''\Phi_S\circ \Phi_{\K'S}= \kappa''_{h_0}\circ \phi''_S\circ \K''\Phi_S\circ \Phi_{\K'S}.
\end{align}
\end{subequations}
The equality between rows \eqref{28d} and \eqref{28e} uses the fact that $\kappa''$ is a natural transformation $\kappa''\colon \K''\Rightarrow \ee_\C$ and that $\ee_\C(\Phi_S)=\Phi_f$.
Since $\kappa''_{h_0}$ is a monomorphism, 
equality \eqref{equation:morphism coherence} holds.
\end{proof}

\begin{remark}
Proposition~\ref{Je streda a uz melu z posledniho.} admits a 
conceptual interpretation. By Remark~\ref{remark:adjunction},
any lax $\oA$-algebra structures $\K', \K''$ 
participate in the adjunctions
$\beta_\C \dashv \K'$ and $\beta_\C \dashv \K''$, and since 
adjoints are unique up
to isomorphism, so are the algebra structures. The connection between 
our results and property-like structures is clarified by the
notion of relative colax-idempotency: we say that a\hbox{ $2$-monad~$\T$} 
with unit $\eta$ on a
$2$-category ${\tt R}$ is colax-idempotent on a full
subcategory ${\tt P}\subseteq {\tt R}$ if, for any lax $\T$-algebra
$a\colon \T(A)\to A $ with $A\in {\tt P}$, there exists a $2$-cell
$\theta \colon \eta_A \circ a \Rightarrow \uu_{\T (A)}$ such that
$(\theta,\uu)$ is the counit-unit pair of an adjunction
$\eta_A\dashv a$. 

Our $2$-monad $\overline{\Ar}$ is colax-idempotent 
on the subcategory  ${\tt PtCat}  \subseteq \Cat_\Ar$ of
pointed categories. It follows that, for lax $\oA$-algebras
$\K'\colon \Ar (\C') \to C'$, $\K''\colon \Ar (\C'') \to \C''$ 
with $\C',\C'' \in {\tt PtCat}$, 
and for a functor $F\colon \C'\to \C''$ in ${\tt PtCat}$, 
there exists a~unique $2$-cell
$\overline{F}\colon F\circ \K' \Rightarrow \K''\circ \Ar (F)$ making
the pair $(F,\overline{F})$ a colax morphism of $\Ar$-algebras.
The unique extension of the 
functor $F$ is
given by the kernel diagram
\[
\xymatrix@R=.5em{
\K'' F(f) \ar@{>->}[rd]^{\kappa''_{F(f)}}
\\
&F(a) \ar[r]^{F(f)} &F(b) \ ,
\\
F(\K' f) \ar@{>->}[ur]_{F(\kappa'_{\!f})} \ar@{-->}@/^1em/[uu]^{\overline{F}_f}
}
\]
with  $f: a \to b \in \C'$.
\end{remark}

\section{The spindle embedding}
\label{Jarka upekla skvely jablecny zavin.}

The aim of this section is to study the relation between the
iterations of the functor $\Ar$ defined in~\eqref{Ztratil jsem klicek od
  kola - porad neco ztracim.}, and the bar construction. Although its
r\^ole is auxiliary, it still contains interesting combinatorial
constructions.  

Let $\underline {\tt n}$ be the 
ordered set $\{0 < 1 < \cdots < n\}$ considered as a category, and
let $\ttN_n (\C) := [\underline
\ttn,\C]$ be the $n$th piece of the categorical version of 
the standard simplicial 
nerve of $\C$. The latter is
the category of chains of composable arrows
\begin{equation}
\label{Mozna odjedu o den pozdeji.}
[f_1,\ldots,f_n] :=
\xymatrix@1{\bullet \ar[r]^{f_1} & \bullet \ar[r]^{f_2} 
& \cdots& \cdots  \bullet \ar[r]^{f_{n-1}}  
& \bullet \ar[r]^{f_n} & \bullet \ .
}
\end{equation}
The morphisms of $\ttN_n (\C)$ are
commutative ladders
\begin{equation}
\label{S Jaruskou jsem byl na motorkach v Nepomuku.}
\xymatrix{\bullet \ar[d]^(.4){\varphi_0} \ar[r]^{f'_1} 
& \bullet \ar[r]^{f'_2}  \ar[d]^(.4){\varphi_1} 
& \cdots& \cdots  \bullet \ar[r]^{f'_{n-1}}  
& \bullet  \ar[d]^(.4){\varphi_{n-1}} \ar[r]^{f'_n} & \bullet \ar[d]^(.4){\varphi_n} 
\\
\bullet \ar[r]^{f''_1} & \bullet \ar[r]^{f''_2} 
& \cdots& \cdots  \bullet \ar[r]^{f''_{n-1}}  & \bullet \ar[r]^{f''_n} & \bullet
}
\end{equation}
in $\C$. We denote by $\ttNw_n(\C) \subset \ttN_n(\C)$, $n \geq
0$, the wide (aka luff)
subcategory with the morphisms as in~(\ref{S Jaruskou jsem byl na motorkach v
  Nepomuku.}), where all $\varphi_i$'s except $\varphi_0$ are the
identities. The natural isomorphisms of categories 
\begin{equation}
\label{Jeste bych se zitra chtel projet.}
\ttNw_a (\ttNw_b(\C)) \cong
\ttNw_{a+b}(\C)
\end{equation} 
can be easily checked.
The collections $\{\, \ttN_n(\C) \mid  n\geq 0\,\}$  and
$\{\, \ttNw_n(\C) \mid n\geq 0\,\}$ with the standard simplicial
operators of the nerve form simplicial objects
$\ttN_\bullet(\C)$ and $\ttNw_\bullet(\C)$ in $\Cat$.

Recall that the (categorical) d\'ecalage is 
the functor $\Dec : \Cat \to \Cat$  that acts
on a category $\C$ as the coproduct
\begin{equation}
\label{Dnes jsme byli s Jarkou na vystave betlemu.}
\Dec(\C) := \bigsqcup_{a \in \C} \C/a .
\end{equation}
The (upper) d\'ecalage  $\dec(S)_\bullet$ of a simplicial object
$S_\bullet$ is the simplicial object with 
\[
\dec(S)_n := S_{n+1},\ n \geq
0,
\] 
with the simplicial operators $d'_i$, $s'_j$ 
given by the simplicial operators of
$d_i$, $s_j$ of $S_\bullet$~as 
\begin{align*}
&d'_i: \dec(S)_n = S_{n+1} \redukce{$\stackrel {d_i}\longrightarrow$} 
S_n = \dec(S)_{n-1}, \ n
\geq 1, \ 0 \leq i \leq n,  \ \hbox { and } 
\\
&s'_j: \dec(S)_n = S_{n+1} \redukce{$\stackrel {s_j}\longrightarrow$} 
S_{n+2}  =   \dec(S)_{n+1},  \ n
\geq 0, \ 0 \leq j \leq n;
\end{align*}
the top operators $d_{n+1} : S_{n+1} \to S_n$ and   
$s_{n+1} : S_{n+1} \to S_{n+2}$ of $S_\bullet$ are forgotten. The relation
between the two d\'ecalages is described in

\begin{proposition}
There is a functorial isomorphism of simplicial objects in $\Cat$
\begin{equation}
\label{Jarce se kupodivu ty motorky moc libily.}
\dec (\ttNW (\C))_\bullet \cong \ttN_\bullet (\Dec(\C)). 
\end{equation}
Here~\,$\ttNW_\bullet(\C)=  \{\, \ttNW_n(\C) \mid  n\geq 0\, \}$ is the
simplicial object formed by still another wide subcategories $\ttNW_n(\C)$
of \, $\ttN_n(\C)$ whose morphisms~\eqref{S Jaruskou jsem byl na
  motorkach v Nepomuku.} have $\phi_n = \id$. 
\end{proposition}

\begin{proof}
The canonical isomorphism between the objects of the categories
in~(\ref{Jarce se kupodivu ty motorky moc libily.}) is obvious and
well-known, cf.\ e.g.\ the notes in the introduction
to~\cite{GarnerKockWeber}. 
Verifying the morphism part and functoriality is straightforward. 
\end{proof}

\begin{proposition}
\label{Prvni dny nam proprsi.}
The functors \,  $\ttNw_n : \Cat \to \Cat$ and\,  $\Dec^n : \Cat \to \Cat$ 
are isomorphic for each $n \geq 0$. 
\end{proposition}

\begin{proof}
We will construct the functorial isomorphisms 
\begin{equation}
\label{Pristi tyden zavody v Jicine.}
\ttNw_n(\C) \cong \Dec^n(\C), \ \C \in \Cat,\ n \geq 0,
\end{equation}
inductively.
There is nothing to prove for $n=0$. Equation~\eqref{Pristi tyden
  zavody v Jicine.} with $n=1$ reads
\begin{equation}
\label{Musim si nakoupit na dva dny dopredu.}
\ttNw_1(\C) \cong \Dec(\C), \ \C \in \Cat,
\end{equation}
which is an obvious isomorphism of categories. Now we apply this
equality to $\ttNw_1(\C)$ in place of $\C$. The result is
\[
\ttNw_1(\ttNw_1 (\C)) \cong \Dec(\ttNw_1(\C)).
\]
The left hand side equals
$\ttNw_2 (\C)$ 
by~(\ref{Jeste bych se zitra chtel projet.}), while the right hand
side is $\Dec^2(\C)$ by~\eqref{Musim si nakoupit na dva dny dopredu.}. 
This verifies~(\ref{Pristi tyden zavody v Jicine.}) for $n = 2$. Iterating this process 
gives~(\ref{Pristi tyden zavody v Jicine.}) for an arbitrary $n$. 
\end{proof}

\begin{example}
The isomorphism in~(\ref{Pristi tyden zavody v Jicine.}) for $n=1$
interprets  a morphism
$f: a \to b \in \ttNw_1(\C)$ as an
element of $\C/b \subset \Dec(\C)$. For $n=2$ it brings   
\redukce{$a \stackrel r\to b \stackrel s\to c \in \  \ttNw_2(\C)$} 
to the triangle
\begin{equation}
\label{Koupil jsem si 40ti procentni slehacku. Mnam!}
\xymatrix@R=1.3em@C=1.3em{a \ar[rr]^r \ar[rd]_{sr}  && b \ar[ld]^s
\\
&c
}
\end{equation}
representing a morphism $sr \to s$ in $\C/c$, that is, as an object
of $\Dec(\C)/s \subset \Dec^2(\C)$.
\end{example}

Consider the transformation $J : \Dec \to \Ar$ of functors in
{\tt Cat} such that $J_\C:\Dec(C) \to \Ar(C)$ is the identity on objects 
and sends a morphism $gf \to g \in \Dec(\C)$ represented by
\[
\xymatrix@R=1.9em@C=1.9em{a \ar[rr]^f \ar[rd]_{gf}  && b \ar[ld]^g
\\
&c
}
\]
to the square
\[
\xymatrix{a \ar[r]^f \ar[d]_{gf}  & b \ar[d]^g
\\
c \ar@{=}[r]&c
}
\]
representing a morphism $gf \to f$ \, in $\Ar(\C)$. 

\begin{definition}
\label{zitra v 8:00 rano}
The {\/\em spindle embedding \/} is the composite
\begin{equation}
\label{V patek bohyne.}
\redukce{$\S_n^w : \ttNw_n(\C) \stackrel \cong\longrightarrow \Dec^n(\C)
\stackrel {J^n_\C}\longrightarrow \Ar^n(\C)$}
\end{equation} 
of the isomorphism in~\eqref{Pristi tyden zavody v Jicine.} and
\[
J^n_\C: = \Ar^{n-1} J_\C \circ \Ar^{n-2} J_{\Dec(\C)} \circ \cdots \circ  
J_{\Dec^{n-1}(\C)}.
\]
\end{definition}

Unfolding Definition~\ref{zitra v 8:00 rano} 
leads to an explicit formula for the value of
the spindle on a chain 
\begin{equation}
\label{Jarka odejela na tyden na chalupu.}
\redukce{$\Phi : a_0 \stackrel{\varphi_1}\longrightarrow a_1
\stackrel{\varphi_2}\longrightarrow \cdots
\stackrel{\varphi_n}\longrightarrow a_n \in \ttNw_n(\C)$}
\end{equation}
which we give below.
Denote as before by  $\dvedve^{\times n}$ 
the poset/category 
\[
\dvedve^{\times n} = \{v = (\Rada \epsilon1n) \  |  \
\epsilon_i \in \{0,1\}, 1 \leq i \leq n \},
\]
with the partial order $v' \preceq v''$ if and only if
$\epsilon'_i \leq \epsilon_i''$ for all $1 \leq i \leq n$. 
The objects of $\Ar^n(\C)$ are functors $\dvedve^{\times n} \to \C$, that is,
commutative diagrams of the shape of the
$1$-skelet of an  $n$-dimen\-sional cube,
with the vertices decorated by the objects of $\C$ and the
edges decorated by the morphisms.
For a vertex $v \in \dvedve^{\times n}$ define 
\[
p(v) = 
\begin{cases}
\max\{i \ | \ \epsilon_i=1\},
&\hbox {if $\epsilon \not= (00\cdots0)$;}    
\\
0, &\hbox {if  $\epsilon = (00\cdots0)$.}    
\end{cases}
\] 

The spindle embedding applied on $\Phi$
in~(\ref{Jarka odejela na tyden na chalupu.}) 
decorates a vertex $v \in \dvedve^{\times n}$ by $a_{p(v)}$,
and an arrow $v' \prec v''$ between two adjacent vertices by the
unique composite  $a_{p(v')} \to  a_{p(v'')}$ of the morphisms
in~(\ref{Jarka odejela na tyden na chalupu.}) if $p(v') < p(v'')$, and
by the identity if  $p(v') = p(v'')$.

\begin{example}
The three initial steps of the spindle embedding  
$\S_n^w: \ttNw_n(\C) \hookrightarrow
\Ar^n(\C)$ are the following.

\noindent 
$\bullet$ The categories  $\ttNw_1(\C)$ and $\Ar^1(\C)$ have the
same objects,  and $\S^w_1$ is the identity on objects.

\noindent 
$\bullet$  $\S^w_2: \ttNw_2(\C) \hookrightarrow
\Ar^2(\C)$ sends \ \redukce{$\xymatrix@1{a_0 \ar[r]^f & a_1 \ar[r]^g & a_2
}$} \ to
\[
\xymatrix{(00) \ar[r]^f \ar[d]_{gf} & (10) \ar[d]^g
\\
(01) \ar[r]^\id & (11)
}
\]
where $(00)$ is decorated by $a_0$, $(10)$ by $a_1$, and $(01), (11)$
by $a_2$.

\noindent 
$\bullet$ $\S^w_3: \ttNw_3(\C) \hookrightarrow
\Ar^3(\C)$  sends \ \redukce{$\xymatrix@1{a_0 \ar[r]^f & a_1 \ar[r]^g &
 a_2 \ar[r]^h & a_3
}$} \ to
\[
\xymatrix@C=2em@R=2em{
&(010)  \ar[dd]|\hole^(.75){\id} \ar[rr]^{h}  &&(011) \ar[dd]^(.2){\id}
\\
(000) \ar[ur]^{gf}\ar[dd]_{f} \ar[rr]^(.7){hgf} &&(001)
\ar[dd]^(.2){\id} 
\ar[ur]^\id
\\
&(110) \ar[rr]|\hole^(.3){h}&&(111)
\\
(100)\ar[ur]^g \ar[rr]^(.4){hg}&&
(101)\ar[ur]^\id
}
\]
with 
$(000)$ decorated by $a_0$, $(100)$ decorated by $a_1$, 
$(010)$ and $(110)$  decorated by $a_2$, and
the remaining vertices decorated by $a_3$.
\end{example}

\begin{proposition}
The spindle embedding extends to a functor\  $\S_n: \ttN_n(\C) \to \Ar^n(\C)$.
\end{proposition}

\begin{proof}
An immediate consequence of the explicit description of the spindle
embedding given above. It would also follow from
Proposition~\ref{Vzhledem ke strasnemu pocasi asi jedu jen na otocku.} below.
\end{proof}

\begin{definition}
Let $\iota : \underline{\tt n} \to \dvedve^{\times n}$ be the 
embedding that sends the path 
$0 \to 1 \to \dvedve \to  \cdots \to n$ \ in $\underline{\tt n}$ to the path
\[
(000\cdots0) \longrightarrow (100\cdots0) \longrightarrow (110\cdots0)
\longrightarrow \cdots \longrightarrow (111\cdots1)
\]
in $\dvedve^{\times n}$. We will call the latter the {\/\em spine\/} of the cube
$\dvedve^{\times n}$. The restriction to the spine defines a~functor $R_n :
\Ar^n(\C) \to \ttN_n(\C)$. 
\end{definition}

It follows from the description of the spindle embedding
given above that
the composite \redukce{$\ttN_n(\C) \stackrel {\S_n}\longrightarrow \Ar^n(\C)
\stackrel {R_n}\longrightarrow \ttN_n(\C)$}  equals the identity.  

\begin{proposition}
\label{Vzhledem ke strasnemu pocasi asi jedu jen na otocku.}
The spindle embedding $\S_n(Y)$ is the right Kan extension of 
the functor
$Y : \underline \ttn \to \C$ along
$\iota$ in the diagram
\[
\xymatrix@C=3em{
\dvedve^{\times n}    \ar@/^1.5em/[rd]^{\S_n(Y)}
\\
\underline{\tt n} \ar@{_{(}->}[u]_\iota \ar[r]^{Y}   & \C.
}
\]
Equivalently, there is a natural isomorphisms
\begin{equation}
\label{Jarka odjizdi na tyden na chalupu.}
[X,\S_n  \!(Y)] \cong [R_n(X),Y]
\end{equation}
for each $X \in \Ar^n(\C)$ and $Y \in \ttN_n(\C)$.
\end{proposition}

Notice that Proposition~\ref{Vzhledem ke strasnemu pocasi asi jedu jen
  na otocku.} implies the functoriality of $\S_n: \ttN_n(\C) \to
\Ar^n(\C)$. 
Indeed,~(\ref{Jarka odjizdi na tyden na chalupu.}) along with $R_n\S_n = \id$ gives
\[
[Y',Y''] =  [R_n\S_n\!(Y'), Y''] \cong [\S_n\!(Y'),\S_n\!(Y'')]. 
\]
The existence of a right Kan extension as 
in Proposition~\ref{Vzhledem ke strasnemu pocasi asi jedu jen na
  otocku.} 
follows from general theorems, so it only remains to prove that it
is given by our explicit formula. Since we will  not need this
fact here, we omit the proof.

\section{Nerves of lax $\Ar$-algebras}
\label{Na tu novou vladu opravdu nemam nerv ani jako kategorie.}

Let $\oA =(\Ar,\mu,\eta)$,
$\mu := \Ar\epsilon : \Ar^2 \to \Ar$, $\eta:=\beta : \id \to \Ar$,  
be the induced monad on
the category $\Acoalg$ of $\Ar$-coalgebras, 
cf.~Section~\ref{Dominik prileti v pondeli.}.  The identity
functor $\id : \Acoalg \to \Acoalg$ is a right $\oA$-functor in the
sense of~\cite[Definition~II.6.3]{EKMM}, cf.~also~\cite{stonek}, 
with the action $\rho : \id \, \Ar = \Ar \to \id$
given by the comonad counit $\epsilon : \Ar \to \id$.  
We can therefore form, for an $\oA$-algebra $\C$ with the action $a:
\oA(\C) \to \C$, the monadic bar construction
$B_\bullet(\id,\oA,\C)$ with $B_n(\id,\oA,\C) :=
\id \Ar^n(\C) = \Ar^n(\C)$~\hbox{\cite[Definition~XII.1.1]{EKMM}}.
The boundary operators are given by
\begin{equation}
\label{Bude to zitra na poletani?}
d^\Ar_i := \begin{cases}
\Ar^i\epsilon \Ar^{n-i-1} & \hbox{for} \  0 \leq i < n,
\\
\Ar^{n-1}a & \hbox{for}\ i=n,
\end{cases}
\end{equation}
and the degeneracies by $s^\Ar_i :=  \Ar^i \beta \Ar^{n-i}$, $0 \leq i \leq
n$. Let us introduce the following version of Jardine's supercoherent structure.

\begin{definition}
\label{definition:almost strict supercoherent}
An almost strict supercoherent structure $X_\bullet$ consists of
categories $X_n$, $n \geq 0$, functors $d_i : X_n\to X_{n-1}$, $0 \leq
i \leq n$, $n \geq 1$, and $s_j : X_n \to X_{n+1}$, $0 \leq j \leq n$,
such that $X_0$ equals the terminal category $\jednajedna$, and the
standard simplicial identities, that is
\begin{subequations}
\label{equation:simplicial identities}
\begin{align}
\label{hromada uhli}
d_id_j  &= d_{j-1}d_i \ \ \text{if } i < j,
\\
\label{simpl ids with s} s_is_j  &= s_{j+1}s_{i} \ \ \text{if } i \leq j, \ \hbox { and}
\\
\label{simpl ids combined} d_i s_j &= 
  \begin{cases*}
s_{j-1}  d_i   &  if $i < j$,  \\
\uu & if $i = j, i=j+1,$ \\ 
      s_j d_{i-1} &  if $i > j+1$. 
\end{cases*}
\end{align}
\end{subequations}
hold, except for~\eqref{hromada uhli} with  $i = n-1$, $j=n$, 
which is replaced  by a natural transformation
\begin{equation}
\label{Vcera jsem opet pichl galusku - ty kapitalisticke nic nevydrzi}
\alpha_n :d_{n-1} d_{n} \Longrightarrow d_{n-1} d_{n-1}
\end{equation}
between the functors $X_n \to X_{n-2}$, $n \geq 2$. We moreover require
that the obvious analogs of the diagrams on pages 115--117
of~\cite{Jardine} that involve 
transformations~(\ref{Vcera jsem opet pichl galusku - ty
  kapitalisticke nic nevydrzi})  commute. 
\end{definition}

\begin{theorem}
\label{Musim se dovalit na nakup.}
Let $(\C,a,\phi)$,  
$a: \Ar(\C) \to \C$, $\phi : a \, \Ar a \Longrightarrow a \, \Ar
\epsilon_\C$, 
be a normalized lax $\overline \Ar$-algebra. 
Then the operators $d_i^\Ar$ and $s_i^\Ar$ defined above 
make $\Ar^\bullet(\C)$ an almost strict supercoherent structure.
\end{theorem}

\begin{proof}
Direct verification; the transformations in~(\ref{Vcera jsem opet
  pichl galusku - ty kapitalisticke nic nevydrzi}) are given as
$\alpha_n : = \Ar^{n-2} \phi$, $n \geq 2$.
\end{proof}

\begin{proposition}
The standard simplicial operators of the nerve
\[
d_i : \ttN_{n} (\C)
\to \ttN_{n-1} (\C) \ \hbox { and } \  s_i : \ttN_n(\C) \to
\ttN_{n+1}(\C)
\] 
satisfy for $1 \leq i \leq n$ 
\begin{subequations}
\begin{align}
\label{Odhodlam se zitra k jizde do Kolina?}
d_{i} &= R_{n-1} d_{i-1}^\Ar \S_{n}, 
\ \hbox { and } \
\\
\label{Vcera jsem dojel do Vraneho nad Vltavou.}
s_{i} &= R_{n+1} s_{i-1}^\Ar \S_{n}.
\end{align}
\end{subequations}
\end{proposition} 

\begin{proof}
It is simple to check directly from the explicit description of
the spindle embedding that the operator $d_{n-1}^\Ar = \Ar^n \epsilon$ removes
from the decorated cube representing an object of $\Ar^{n}(\C)$ all
vertices $v=(\epsilon_0 \epsilon_1 \cdots \epsilon_{n-1})$
with  $\epsilon_{n-1} = 1$, as well as the adjacent edges. 
This operation obviously removes the last segment of  the decorated spine
\begin{equation}
\label{Do nedele by melo byt rozumne.}
\Phi : (00\cdots000) \stackrel{\varphi_1}\longrightarrow (10\cdots000)
\stackrel{\varphi_2}\longrightarrow \cdots 
\stackrel{\varphi_{n-1}}\longrightarrow (11\cdots100)
\stackrel{\varphi_{n}}\longrightarrow (11\cdots110),
\end{equation}
which is precisely what $d_{n}$ does to the nerve
$\ttN_{n}(\C)$. This proves~(\ref{Odhodlam se zitra k jizde do
  Kolina?}) for $i=n$.

Likewise $d^\Ar_{i-1} = \Ar^{i-1}\epsilon \Ar^{n-i}$ 
removes all vertices  $v=(\epsilon_0 \epsilon_1
\cdots \epsilon_{n-1})$ with  $\epsilon_{i-1} =
1$. The effect on the
decorated spine~(\ref{Do nedele by melo byt rozumne.})
is composing
$\varphi_{i+1}$ with $\varphi_{i}$, which is precisely what $d_i$
does to
$\ttN_{n}(\C)$. This 
proves~(\ref{Odhodlam se zitra k jizde do Kolina?})
 for the remaining $i$'s, $1 \leq i < n$.

Let us move to~(\ref{Vcera jsem dojel do Vraneho nad Vltavou.}).
The crucial observation is that $\beta \Ar: \Ar(\C) \to
\Ar^2(\C)$ sends $f: a \to b $ of~$\Ar(\C)$~to
\[
\xymatrix@R=1.4em{a \ar[r]^f \ar[d]_f  &b \ar[d]^\id 
\\
b \ar[r]^\id  &b
} \raisebox{-1.4em}{ \ $\in \Ar^2(\C)$.}
\] 
This gives, for $1 \leq i \leq  n$, a description of  
$s^\Ar_{i-1} =  \Ar^{i-1} \beta \Ar^{n-i+1} =  \Ar^{i-1} (\beta \Ar) \Ar^{n-i}$
that uses a map  
of the vertices of the cube $\dvedve^{\times (n+1)}$ to the vertices of the
cube $\dvedve^{\times n}$ defined as follows.
The vertex $w = (\eta_0,\ldots,\eta_{n})$ of $\dvedve^{\times (n+1)}$  
is mapped to the vertex $w_{i-1} = (\Rada
\epsilon0{n-1})$ of $\dvedve^{\times n}$ with 
\[
\epsilon_r =
\begin{cases}
\eta_r&\hbox {if $0 \leq r < i-1$,}
\\
\max\{\eta_{i-1},\eta_{i}\}&\hbox {if $r = i$, and}
\\
\eta_{r+1}&\hbox {if $i-1 < r \leq n-1$.}
 \end{cases}
\] 
Interpreting as before the objects of $\Ar^{n}(\C)$ and  
$\Ar^{n+1}(\C)$ as decorated cubes, then   \hbox{$x \in \Ar^{n}(\C)$} 
determines $s^\Ar_{i-1}(x) \in \Ar^{n+1}(\C)$  such that the vertex $w$ of
$\dvedve^{\times n}$ inherits the decoration of the corresponding vertex
$w_{i-1}$ of $\dvedve^{\times n}$, and the map between the vertices $w'$ and
$w''$ of $\dvedve^{\times (n+1)}$ is the map between the vertices $w_{i-1}'$ and
$w''_{i-1}$ in $\dvedve^{\times n}$ if such a map exists, otherwise there is no
map. The rule $x \mapsto s^\Ar_{i-1}(x)$ describes the
action the map $s^\Ar_{i-1} : \Ar^n(\C)  \to  \Ar^{n+1}(\C)$.
It is easy to check that the effect on the decorated spine~\eqref{Do
  nedele by melo byt rozumne.} is inserting the identity between
$\varphi_i$ and $\varphi_{i+1}$, which is precisely what the
degeneracy $s_i$ does to $\ttN_n(\C)$. 
\end{proof}

We extend the definition of the  boundary operators in~(\ref{Bude
  to zitra na poletani?}) by introducing an auxiliary operator $\dminus : \Ar^n(\C) \to
\Ar^{n-1}(\C)$ for $n \geq 1$. 
To this end we interpret the objects of $\Ar^{n}(\C)$
as morphisms $S : X \to Y$ in $\Ar^{n-1}(\C)$ and define $\dminus(S)$
to be the target of $S$,
\begin{subequations}
\begin{equation}
\label{V nedeli se ochladi.}
d^\Ar_{-1}(S) = Y \in \Ar^{n-1}(\C).
\end{equation} 
Likewise, we define $\sminus: \Ar^n(\C) \to \Ar^{n+1}(\C)$ by 
\begin{equation}
\label{Bude dnes bourka?}
\sminus(X) := X \stackrel \id\longrightarrow X \in \Ar^{n+1}(\C)
\end{equation}
\end{subequations}
for $X \in \Ar^n(\C)$, $n \geq 0$.

\begin{proposition}
The operators in~(\ref{V nedeli se ochladi.}) and~(\ref{Bude dnes
  bourka?}) satisfy the equalities
\begin{subequations}
\begin{align}
\label{Bude asi 35 stupnu, juj!}
d_{n-1}^\Ar d^\Ar_{-1} &= d^\Ar_{-1} d_n^\Ar, \ n \geq 2,
\\
\label{Mozna jsem mel jet uz}
d^\Ar_{n+1} s^\Ar_{-1} &=  s^\Ar_{-1} d^\Ar_n, \ n \geq 1,
\\
\label{vcera.}
s^\Ar_{n+1} s^\Ar_{-1} &=  s^\Ar_{-1}s^\Ar_n, \ n \geq 0.
\end{align}
\end{subequations}
The operators in~(\ref{V nedeli se ochladi.}) and~(\ref{Bude dnes
  bourka?}) are, moreover, related with the standard simplicial operators
$d_0  : \ttN_n(\C) \to \ttN_{n-1}(\C)$ and $s_0:
\ttN_n(\C) \to \ttN_{n+1}(\C)$ of the bar construction \, $\ttN_\bullet(\C)$ by
\begin{equation}
\label{mne to ceka}
d_0 = R_{n-1} \dminus \S_{n}, \  n \geq 1, \ \hbox { and } 
\ s_0 = R_{n+1} \sminus \S_{n}, \ n \geq 0.
\end{equation}
\end{proposition}

\begin{proof}
Equation~(\ref{Bude asi 35 stupnu, juj!})
follows from the functoriality of $d_n^\Ar$ and 
$d^\Ar_{n-1}$. Indeed,  
let  $S: X \to Y$ be a morphism in $\Ar^{n-1}(\C)$
representing an object of $\Ar^n(\C)$.
Then $d^\Ar_n(S)$ is the induced map $d^\Ar_{n-1}(S) : d^\Ar_{n-1}(X) \to
d^\Ar_{n-1}(Y)$, so  $d_{-1}^\Ar d^\Ar_n(S)= d^\Ar_{n-1}(Y)$ by~(\ref{V nedeli se
  ochladi.}). On the other hand, $d^\Ar_{-1}(S) = Y$, so $d^\Ar_{n-1} 
\dminus(S) = d_{n-1}^\Ar(Y)$, which 
verifies~(\ref{Bude asi 35 stupnu, juj!}).
The equalities~(\ref{Mozna jsem mel jet uz})
and~(\ref{vcera.}) follow in the same manner from the functoriality of
$d^\Ar_{n}, d^\Ar_{n+1}, s_n^\Ar$ and $s^\Ar_{n+1}$.
We leave the details to the reader. Equation~(\ref{mne to ceka}) easily
follows from the definition of the operators $\dminus$ and $\sminus$.
\end{proof}

\begin{corollary}
\label{corollary:Martinuv supervysledek}
Under the assumptions of Theorem~\ref{Musim se dovalit na nakup.},
the nerve\, $\ttN_\bullet(\C)$ is the upper d\'ecalage of an almost
strict supercoherent
structure \, ${\widetilde \ttN}_\bullet(\C)$. 
\end{corollary}

\begin{proof}
The supercoherent structure 
${\widetilde \ttN}_\bullet(\C)$ is defined as
$\ttN_{\bullet-1}(\C)$, extended by setting
${\ttN}_{-1}(\C) : = {\jednajedna}$, the terminal
category. The additional simplicial operators are defined by
\begin{subequations}
\begin{align}
\label{Dam si zmrzlinu.}
d_n &:= R_{n-2} d^\Ar_{n-1} \S_{n-1} : \ttN_{n-1}(\C) \longrightarrow
\ttN_{n-2}(\C)  \hbox{ for $n \geq 2$, and }
\\
\label{A pak jeste kornout!}
s_n &:= R_n s_{n-1}^\Ar \S_{n-1}
:  \ttN_{n-1}(\C) \longrightarrow
\ttN_{n}(\C) \hbox { for $n \geq 1$}.
\end{align}
\end{subequations}
The operators $d_i: {\ttN}_0(\C) \to
{\ttN}_{-1}(\C)$, $i = 1,2$, are the unique morphisms to the terminal
category, and $s_0 : {\ttN}_{-1}(\C) \to
{\ttN}_{0}(\C)$ sends the unique object of ${\jednajedna}$ to
the chosen terminal object~$0$ of
$\C$.

The spindle embedding identifies $\ttN_\bullet(\C)$ with its image
in $\Ar^\bullet(\C)$. It follows from~(\ref{Odhodlam se zitra k jizde
  do Kolina?})--(\ref{Vcera jsem dojel do Vraneho nad Vltavou.}), 
and the defining equations~(\ref{Dam si zmrzlinu.})--(\ref{A pak
  jeste kornout!}) 
that, under this identification, the operators
\begin{equation}
\label{Jake bude pocasi?}
\Rada d1n : \ttN_{n-1}(\C) \longrightarrow  \ttN_{n-2}(\C), n
\geq 2,
\ \hbox{and} \
\Rada s1n : \ttN_{n-1}(\C) \longrightarrow  \ttN_{n}(\C), \ n \geq 1,
\end{equation} 
are the restrictions of the operators
\[
\rada {d_0^\Ar}{d_{n-1}^\Ar} : \Ar^{n-1}(\C) \longrightarrow  
\Ar^{n-2}(\C), \ n \geq 2,
\ \hbox{and} \
\rada {s_0^\Ar}{s_{n-1}^\Ar} : \Ar^{n-1}(\C) \longrightarrow
\Ar^{n}(\C), 
n \geq 1.
\] 
This shows that the extended nerve ${\widetilde
  \ttN}_\bullet(\C) = \ttN_{\bullet-1}(\C)$
is an almost strict supercoherent structure, with the 
transformations in~(\ref{Vcera jsem opet pichl galusku - ty kapitalisticke
  nic nevydrzi}) given by 
\[
\alpha_n := R_{n-3} \alpha_{n-1}^\Ar \S_{n-1} : d^\Ar_{n-2} d^\Ar_{n-1} 
\longrightarrow  d^\Ar_{n-2} d^\Ar_{n-2}, \     n \geq 3,
\] 
where $\alpha^\Ar_{n-1} :d_{n-2}^\Ar d_{n-1}^\Ar \longrightarrow d_{n-2}^\Ar d_{n-2}^\Ar$
are the constraints of the almost strict 
supercoherent structure $\Ar^\bullet(\C)$.

We need also to  verify that $d_0$ and
$s_0$ satisfy the required relations with respect to the operators
in~(\ref{Jake bude pocasi?}). Since $\Rada d1{n-1}$ and $\Rada
s1{n-1}$ are the standard simplicial operators of the bar
construction, we only need to explore the relations that involve
$d_n : \ttN_{n-1}(\C) \to  \ttN_{n-2}(\C)$, $n \geq
2$, and  $s_n : \ttN_{n-1}(\C) \to  \ttN_{n}(\C)$, $n \geq 1$.
However,  they follow
from~(\ref{Bude asi 35 stupnu, juj!})--(\ref{vcera.}) and the fact
that $d_0$ and $s_0$ are the restrictions of $\dminus$ and $\sminus$
to $\ttN(\C)_\bullet$ by~(\ref{mne to ceka}). 

The last thing to be checked are the relations/identities involving \,
${\widetilde \ttN}_0(\C) =  \ttN_{-1}(\C)$ as the target. They
are however satisfied automatically, since ${\widetilde
  \ttN}_0(\C)$ is the terminal category by definition. 
\end{proof}

The creation of the extended simplicial structure on the
nerve $\ttN_\bullet(\C)$ is captured by the diagram
\[
\xymatrix{\Ar^{n-1}(\C) : \hspace{-1em}& \ar@{|->}[d]   \dminus &\ar@{|->}[d]  d^\Ar_0 
&\ar@{|->}[d]   d^\Ar_1 & \cdots & d^\Ar_{n-2} \ar@{|->}[d] &\ar@{|->}[d] 
\  d^\Ar_{n-1} .
\\
{\widetilde \ttN}_n(\C) 
\ar@{^{(}->}[u]_{\S_n}:  \hspace{-1.8em} 
& d_0 & d_1 &  d_2 & \cdots & d_{n-1} & 
  d_{n}
}
\]
In that diagram, $\widetilde \ttN_n(\C)$ is identified with a subcategory of
$\Ar^{n-1}(\C)$ via the spine embedding, and the arrows $\longmapsto$ denote
restrictions. We emphasize that $\dminus$ and $\sminus$ 
does not extend the simplicial structure of $\Ar^\bullet(\C)$. 
We used these auxiliary operators 
since they satisfy~(\ref{Bude asi 35 stupnu, juj!})--(\ref{vcera.})
and restrict to $d_0$ and $s_0$ by~(\ref{mne to ceka}).  

\section{Simplicial characterization of kernels}
\label{Jarka si koupila novy mobil.}

The iteration of the functor $\Ar$ defined in~\eqref{Ztratil jsem
  klicek od kola - porad neco ztracim.} produces a calculus of arrows, squares
and cubes, while the functorial nerve gives a calculus of
arrows, triangles, and simplices. This section adapts
Theorem~\ref{theorem:kernels} to the simplicial setting.  So, let $\C$ be a
pointed category and $\ttN_\bullet(\C)$ its functorial simplicial nerve
recalled in Section~\ref{Jarka upekla skvely jablecny zavin.}. 
The zero object $0$ extends the nerve by
\begin{enumerate}
\item[--] 
the terminal category $\NN_{-1}(\C):=\{0\}$,
\item[--] 
the functor $s_0\colon\NN_{-1}(\C) \to \NN_{0}(\C)$ which picks 
the zero object,
\item[--] 
the functors $d_0,d_1\colon \NN_0(\C) \to \NN_{-1}(\C)$ defined as 
the terminal functors to
$\NN_{-1}(\C)=\{0\}$, and
\item[--] 
the additional functors 
$ s_{n+1}\colon \NN_{n}(\C) \to \NN_{n+1}(\C)$, $n\geq 0$, 
given by the post-composition with the terminal morphism to $0$, i.e.~sending~\eqref{Mozna odjedu o den pozdeji.} to
\[[f_1,\ldots ,f_n ,\,!\,]=
\xymatrix@1{\bullet \ar[r]^{f_1} & \bullet \ar[r]^{f_2} 
& \cdots& \cdots  \bullet \ar[r]^{f_{n-1}}  
& \bullet \ar[r]^{f_n} & \bullet \ar[r]^{!}   \ & \ 0 
} \in \ttN_{n+1}(\C).
\] 
\end{enumerate}
The result is the structure
\[
\xymatrix@C=4.86em@R=8em{\ttN_{-1}(\C) 
\ar@{->}[r]|-{s_0} \ar@/^1em/@{<-}[r]|-{d_1}   
 \ar@/_1em/@{<-}[r]|-{d_0}   
&\ttN_{0}(\C) 
\ar@{<-}[r]|-{d_1} \ar@/^1em/[r]|-{s_1}  \ar@/_1em/[r]|-{s_0} 
\ar@/_2em/@{<-}[r]|-{d_0}    
& \ttN_{1}(\C)  
\ar[r]|-{s_1} \ar@/^1em/@{<-}[r]|-{d_2}   
 \ar@/_1em/@{<-}[r]|-{d_1}  \ar@/^2em/[r]|-{s_2}  
\ar@/_2em/[r]|-{s_0}  
 \ar@/_3em/@{<-}[r]|-{d_0}   
&\ttN_{2}(\C)  \ar@{}[r]|{\hbox {\Large $\ldots$}} &
}
\]
It can be easily checked that these new operators are 
compatible with the rest of simplicial
operators of the nerve, i.e.~that they satisfy the standard simplicial
identities recalled in~\eqref{equation:simplicial identities}
whenever they make sense.
We define a family of natural transformations
\[
\nu\colon \uu_{\NN_{n+1}(\C)}\Longrightarrow s_{n+1} d_{n+1}, \ n \geq 0,
\]
between the functors $\NN_{n+1}(\C)\to \NN_{n+1}(\C)$ as
follows. For an object $T=[f_1,\ldots,f_{n+1}]\in\NN_{n+1}(\C)$ the
corresponding component $\nu_T$ is the morphism
\[
\nu_T=(\uu,\ldots,\uu,\ !\ )\colon[f_1,\ldots,f_{n+1}] \to
[f_1,\ldots,f_n,\ !\ ]\] 
in $\NN_{n+1}(\C)$, that is the morphism
\begin{equation}
\label{Na Zimni skole spousta snehu a ja tam nebudu.}
\raisebox{-1.7em}{$\nu_T=$} \hskip 1.5em
\xymatrix{
\bullet \ar[r]^{f_1} \ar[d]^{\id}  & \ar[d]^{\id} \bullet \ar[r]^{f_2}  
&\ar[d]^{\id}  \bullet
\ar@{}[r]|{\hbox{\large $\ldots$}}
&\bullet\ar[d]^{\id}  \ar[r]^{f_n} &\bullet\ar[d]^{\id}
\ar[r]^{f_{n+1}}
&\bullet\ar[d]^{!}  
\\
\bullet\ar[r]^{f_1} & \bullet\ar[r]^{f_2} & \bullet
\ar@{}[r]|{\hbox{\large $\ldots$}} &\bullet\ar[r]^{f_n} &\bullet\ar[r]^{!} &0 \tecka
}
\end{equation}
\begin{proposition}
\label{proposition:kernels_as_functor_II}
Let \/ $\C$ be a pointed category and let there be 
two additional functors
\begin{subequations}
\begin{equation}
\label{Greta mne spolehlive dovezla.}
d_3\colon \NN_2(\C)\longrightarrow\NN_1(\C) \ \text{ and } \ d_2\colon
\NN_1(\C)\longrightarrow\NN_0(\C),
\end{equation} 
compatible with the simplicial operators considered above,
i.e.~satisfying the equations 
in~\eqref{equation:simplicial identities}. If~further
\begin{equation}
\label{Mozna pojedu az v utery.}
d_2\big(d_3(\nu_{s_1(f)})\big)=\uu_{d_2(f)},
\end{equation}
\end{subequations}
then
\[
\kappa_f:= d_2(\nu_f)\colon d_2(f)\longrightarrow a,
\] 
where $f: a \to b$ is a morphism of \/ $\C$,
equips $\C$ with kernels.
\end{proposition}

The structure in the proposition therefore involves the operators as in
\[
\xymatrix@C=4.86em@R=8em{\ttN_{-1}(\C) 
\ar@{->}[r]|-{s_0} \ar@/^1em/@{<-}[r]|-{d_1}   
 \ar@/_1em/@{<-}[r]|-{d_0}   
&\ttN_{0}(\C) 
\ar@{<-}[r]|-{d_1} \ar@/^1em/[r]|-{s_1}  \ar@/_1em/[r]|-{s_0} 
\ar@/^2em/@{<-}[r]|-{d_2}    
\ar@/_2em/@{<-}[r]|-{d_0}    
& \ttN_{1}(\C)  
\ar[r]|-{s_1} \ar@/^1em/@{<-}[r]|-{d_2}   
 \ar@/_1em/@{<-}[r]|-{d_1}  \ar@/^2em/[r]|-{s_2}  
\ar@/_2em/[r]|-{s_0}  \ar@/^3em/@{<-}[r]|-{d_3}
 \ar@/_3em/@{<-}[r]|-{d_0}   
&\ttN_{2}(\C)  \ar@{}[r]|{\hbox {\Large $\ldots$}} &
}
\]
It will be convenient for the proof of
Proposition~\ref{proposition:kernels_as_functor_II} to notice that the value
$\kappa_{[g,f]}:=d_3(\nu_{[g,f]})$ 
at a composable pair $[g,f]\in \NN_2(\C)$ is
given by the commutative square
\begin{equation}
\label{Pozitri jedu zase do Bonnu.}
\raisebox{-2em}{$\kappa_{[g,f]}=$} \hskip 1em
\xymatrix@C=3em{
d_2(fg) \ar[r]^{d_3[g,f]}
\ar[d]_{\kappa_{fg}}   & d_2(f) \ar[d]^{\kappa_f}
\\
a \ar[r]^g   & b .
}
\end{equation}
The vertical edges $\kappa_{fg}$ and $\kappa_{f}$ are determined by
the simplicial
identities, so $\kappa_{[g,f]}$ is a morphism 
\[
\kappa_{[g,f]}=(\kappa_{fg},\kappa_{f})\colon d_3[g,f]\to g.
\]
We will also need the for 
the proof of Proposition~\ref{proposition:kernels_as_functor_II} the following

\begin{lemma} 
\label{lemma:kappa_0_retr}
For any zero morphism $0\colon a\to b$, the morphism $\kappa_0\colon
d_2(0)\to a$ has a retraction, i.e.~there is a morphism $r\colon a\to
d_2(0) $ such that  $\kappa_0r=\uu_a$.
\end{lemma}

\begin{proof}
The retraction $r$ is given as $r: =d_2(P)$ for the morphism
$P=(\uu_a,!^b)\colon !_a\to 0$, represented by the left square of the
diagram below. Applying 
the functoriality of $d_2\colon \NN_1(\C)\to \NN_0(\C)$ the
following composite of squares~$P$ (left) and $\nu_0$ (right) in
$\NN_1(\C)$
\[
\xymatrix@R=1em@C=1.2em{
a \ar[dd]_{!_a} 
\ar[rr]_{\id_a} \ar@/^1.5em/[rrrr]^{\id_a}   && \ar[dd]_0
\ar[rr]_{\id_a} a &&a \ar[dd]^{!_a}
\\
\ar@{}[rr]|P  && \ar@{}[rr]|{\nu_0}&&
\\
0 \ar[rr]^{!^b}\ar@/_1.5em/[rrrr]_{\id_0} && b\ar[rr]^{!_b} && 0
}
\]
we compute that
\[
\kappa_0\circ r=d_2(\nu_0)\circ d_2(S)=d_2(\nu_0\circ
S)=d_2(\uu_{!_a})=\uu_{d_2(!_a)}=\uu_a,
\] 
so $r$ is indeed a retraction of $\kappa_0$.
\end{proof}

\begin{proof}[Proof of Proposition~\ref{proposition:kernels_as_functor_II}.] 
We will show that the morphisms $\kappa_f\colon d_2(f)\to a$ give kernels.  First observe, that the simplicial identities
give 
\[
d_2(\uu_a)=d_2(s_0(a))=s_0(d_1(a))=s_0(0)=0,
\] 
and hence
$\kappa_{\uu_a}=\ !^a\colon 0\to a$. Similarly, $d_2(!_a)=d_2(s_1(a))=a$,
and, since the functor $d_2$ preserves
identities,
it~follows that  
\[
\kappa_{!_a}=d_2(\uu_{!_a})=\uu_{d_2(!_a)}=\uu_a.
\]
This gives the two normalizing conditions in (ii) of
Definition~\ref{definition:algebraic_kernels}.
For any morphism $f\colon a \to b$, the equality
$f\kappa_f=0$ required by item (i) of
Definition~\ref{definition:algebraic_kernels}  follows from the commutativity of the square~$\kappa_{[f,\uu_b]}$ in
\[
\xymatrix@R=.7em{
d_2(f) \ar[r]^(.55){\kappa_f} \ar[dd]_{!} & a  \ar@/^/[rd]^f \ar[dd]^f
\\
\ar@{}[r]|{\kappa_{[f,\uu_b]}}&& b \tecka
\\
0 \ar[r]^{!}&b \ar@/_/[ur]_{\id_b}
}
\]

Let us attend to item (i) of
Definition~\ref{definition:algebraic_kernels}. Assume therefore 
that $g\colon c\to a$ is a morphism with $fg=0$. The induced morphism
$\Tilde{g}\colon c\to d_2(f)$ to the kernel is given by the
commutative square $\kappa_{[g,f]}$ in~\eqref{Pozitri jedu zase do
  Bonnu.} as the morphism $d_3[g,f]$ pre-composed
with the retraction $r$ of Lemma~\ref{lemma:kappa_0_retr}, that is
\[
\Tilde{g}:=d_3[g,f]\circ r.
\]
In other words, $\Tilde g$ is the diagonal of the square in the diagram
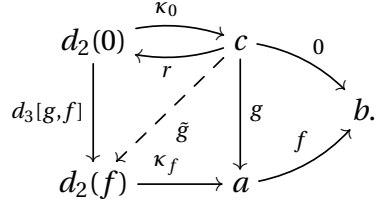
\[
\xymatrix@R=1em@C=3em{
d_2(0) \ar@/^/[r]^{\kappa_0} \ar[dd]_{d_3[g,f]}   
& c \ar@{-->}[ldd]^{\Tilde g}    \ar@/^/[l]^r  \ar@/^/[rd]^0 \ar[dd]^g   &
\\
&& b.
\\
d_2(f)   \ar[r]^{\kappa_f} & a \ar@/_/[ur]^f
}
\]
The equality $\kappa_f {\Tilde g} = g$ is easy to verify, since
\[
\kappa_f \circ\Tilde{g} = \kappa_f\circ d_3[g,f] \circ r=g \circ\kappa_0\circ r=g.
\]
To show the uniqueness of $\Tilde{g}$, note first that for an object
$f\colon a\to b\in\NN_1(\C)$, $s_1(f)$ is the composable pair
$s_1(f)=[f,\uu_b]=a\xrightarrow{f} b\xrightarrow{\uu}b$ and, by~(\ref{Mozna pojedu az v utery.}),
\begin{equation}
\label{Jarka opet upekla zavin, vezmu si ho na cestu.}
d_2(\kappa_{[f,\id_b]})=d_2(\kappa_{s_1(f)})=d_2(d_3(\nu_{s_1(f)}))=\uu_{d_2(f)},
\end{equation}
which is recorded by the following diagram.
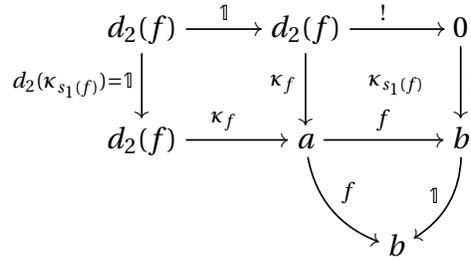
\[
\xymatrix@C=1em{
d_2(f) \ar[rr]^\id \ar[d]_{d_2(\kappa_{s_1(f)}) = \id}  
&& d_2(f)  \ar[d]_{\kappa_f}  \ar[rr]^{!} 
&\ar@{}[d]|{\kappa_{s_1(f)}}&0 \ar[d]^{!}
\\
d_2(f)  \ar[rr]^{\kappa_f}   && \ar[rr]^f  a  \ar@/_.9em/[rd]^f   
&& b  \ar@/^.8em/[ld]_\id 
\\
&&&b
}
\]
Let $h\colon c\to d_2(f)$ be another morphism with $\kappa_fh=g$. The
uniqueness $h=\Tilde{g}$ follows from the commutativity of the
leftmost square of the diagram
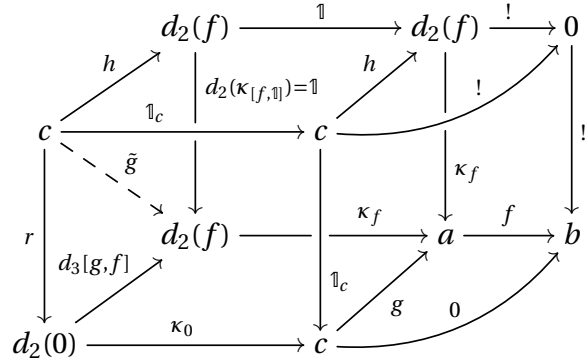
\[
\xymatrix{
& d_2(f) \ar[rr]^{\id}  \ar[dd]^(.3){d_2(\kappa_{[f,\id]}) = \id}|\hole   && d_2(f) \ar[dd]^(.7){\kappa_f}|(.45)\hole \ar[r]^{!} 
&0 \ar[dd]^{!}
\\
c \ar@{-->}[rd]^{\Tilde g} \ar[dd]_r \ar[ru]^h \ar[rr]^(.4){\id_c} && c \ar[dd]^(.7){\id_c}
\ar@/_1.4em/[rru]^(.6){!}    \ar[ru]^h
\\
& d_2(f)\ar[rr]^(.7){\kappa_f}|\hole && a \ar[r]^f & b
\\
d_2(0)\ar[rr]^{\kappa_0} \ar[ur]^(.55){d_3[g,f]} 
&& c\ar@/_1.4em/[rru]^{0} \ar[ur]_g
}
\] 
combined with~(\ref{Jarka opet upekla zavin, vezmu si ho na cestu.}).
This square is the value 
$d_3(\uu_c,\kappa_f,!)$ of $d_3$ on the morphism
\[
(\uu_c,\kappa_f,!)\colon [h,\ !\ ]\longrightarrow [g,f] 
\] 
in~$\NN_2(\C)$, represented by the prism on the right of that diagram.
We conclude that
\[
d_2(\kappa_{[f,\id]}) \circ h=d_3[g,f]\circ r=\Tilde{g}.\qedhere
\]
\end{proof}

\begin{theorem}
\label{Jen aby si stihla vyridit pas.}
A pointed category $\C$ has kernels if and only
if there is an isomorphism 
\[
\ttN_\bullet(\C) \cong\dec(X)_\bullet \, ,
\]
where $X$ is an almost strict supercoherent structure. 
\end{theorem}

Before proving the theorem we formulate a necessary  auxiliary material.
The situation when $\ttN_\bullet(\C) \cong\dec_\bullet(X)$ amounts to
the scheme
\begin{equation}\label{NC=decX scheme}
\xymatrix@C=4.86em@R=7em{X_0 \ar@{->}[r]|-{s_0} \ar@/^1em/@{<-}[r]|-{d_1}   
 \ar@/_1em/@{<-}[r]|-{d_0}   
&X_1 \ar@{<->}[d]^\cong 
\ar@{<-}[r]|-{d_1} \ar@/^1em/[r]|-{s_1}  \ar@/_1em/[r]|-{s_0} 
\ar@/^2em/@{<-}[r]|-{d_2}    \ar@/_2em/@{<-}[r]|-{d_0}    
& X_2 \ar@{<->}[d]^\cong 
\ar[r]|-{s_1} \ar@/^1em/@{<-}[r]|-{d_2}   
 \ar@/_1em/@{<-}[r]|-{d_1}  \ar@/^2em/[r]|-{s_2}  
\ar@/_2em/[r]|-{s_0}  \ar@/^3em/@{<-}[r]|-{d_3}
 \ar@/_3em/@{<-}[r]|-{d_0}   
&X_3  \ar@{<->}[d]^\cong 
\ar@{<-}[r]|-{d_2} \ar@/^1em/[r]|-{s_2}  \ar@/_1em/[r]|-{s_1} 
\ar@/^2em/@{<-}[r]|-{d_3}    \ar@/_2em/@{<-}[r]|-{d_1}
\ar@/^3em/[r]|-{s_3}  \ar@/_3em/[r]|-{s_0} 
\ar@/^4em/@{<-}[r]|-{d_4} \ar@/_4em/@{<-}[r]|-{d_0}
&X_4 \ar@{<->}[d]^\cong 
\ar@{}[r]|{\hbox {\Large $\ldots$}} 
&
\\
&
\ttN_0(\C) \ar@{->}[r]|-{s_0} \ar@/^1em/@{<-}[r]|-{d_1}   
 \ar@/_1em/@{<-}[r]|-{d_0}   
&\ttN_1(\C)
\ar@{<-}[r]|-{d_1} \ar@/^1em/[r]|-{s_1}  \ar@/_1em/[r]|-{s_0} 
\ar@/^2em/@{<-}[r]|-{d_2}    \ar@/_2em/@{<-}[r]|-{d_0}    
&\ttN_2(\C) 
\ar[r]|-{s_1} \ar@/^1em/@{<-}[r]|-{d_2}   
 \ar@/_1em/@{<-}[r]|-{d_1}  \ar@/^2em/[r]|-{s_2}  
\ar@/_2em/[r]|-{s_0}  \ar@/^3em/@{<-}[r]|-{d_3}
 \ar@/_3em/@{<-}[r]|-{d_0}   
&\ttN_3(\C) \ar@{}[r]|{\hbox {\Large $\ldots$}} 
&
}
\end{equation}
The vertical isomorphisms equip the nerve of $\C$ with additional operators
$$d_{n+1} \colon \ttN_n(\C) \to \ttN_{n-1}(\C),\ n \geq 1,$$
$$s_{n+1}\colon
\ttN_n(\C) \to \ttN_{n+1}(\C), n \geq 0,$$ which are fully compatible
with the standard simplicial operators of the nerve. We will need
the following   

\begin{lemma}
\label{Vzal jsem si ze sklepa koste, lopatku ale nemam.}
Let $X_\bullet$ be an almost strict supercoherent structure. 
Suppose moreover that \/ $\C$ is a category such that
\/ $\ttN_\bullet(\C) \cong \dec_\bullet(X)$. Then, for each $n \geq 0$,
\begin{equation}
\label{Opet v Bonnu.}
s_{n+1}(A_0 \stackrel{f_1} \longrightarrow A_1  \longrightarrow \cdots
\longrightarrow
A_{n-1}
\stackrel{f_n} \longrightarrow A_n)
=
A_0 \stackrel{f_1} \longrightarrow A_1  \longrightarrow \cdots
\longrightarrow
A_{n-1} \stackrel{f_n} \longrightarrow A_n \stackrel!\longrightarrow U,
\end{equation}        
where $U$ is a fixed terminal object of \/ $\C$.
\end{lemma}

\begin{proof}
Let us denote $U := s_0(*)$, where $*$ is the unique object of the terminal
category $X_0$. For $A \in \C = \ttN_0(\C)$, $s_1(A) \in \ttN_1(\C)$
is of the form $R \xrightarrow {s_A} T \in \ttN_1(\C)$. Using
$d_0s_1=s_0d_0$ we obtain
\[
T=d_0(s_1(A)) = s_0(d_0(A)) = s_0(*) = U.
\]
Since $d_1s_1 = \id$, we likewise obtain 
\[
A = d_1(s_1(A)) = d_1(R \xrightarrow {s_A} T) = R,
\]
so, for each $A \in \C$, 
\begin{equation}
\label{Po dlouhe dobe poradna zima a ja jsem v Bonnu jako pitomec.}
s_1(A) = A  \stackrel {s_A} \longrightarrow U
\end{equation}
for some map $s_A$ determined by $A$.
We moreover obtain
\[
s_1(U) = s_1(s_0(*)) = s_0(s_0(*)) = s_0(U) = U \stackrel {\id_U}
\longrightarrow U,
\]
so $s_U = \id_U$ -- the normalization.
Let us move on. Since $d_2s_2 = \id$, we see that $s_2(A \xrightarrow {f}
B)$ is of the form $A \stackrel {f}\to B  \to Y$ for an
unspecified morphism $B \to Y$.
The equality $d_0s_2=s_1d_0$ combined with~(\ref{Po dlouhe dobe
  poradna zima a ja jsem v Bonnu jako pitomec.}) however tell us that
this morphism equals
\[
d_0\big(s_2(A \stackrel {f}\longrightarrow
B)\big) = s_1\big(d_0(A \Xrightarrow {f}
B)\big) = s_1(B) = B  \stackrel {s_B} \longrightarrow  U,
\]
so $s_2(A \xrightarrow f B) = A \xrightarrow f B \xrightarrow {s_B} U$.
The identity $ d_1s_2 =  s_1d_1$ gives
\[
s_B \circ f = d_1\big(s_2(A \Xrightarrow f B)\big) = s_1\big(d_1 (A \Xrightarrow f B)\big)
= s_A. 
\]
In the particular case when $f : A \to U$ the above equations give
$s_U \circ f = s_A$ which, since $s_U = \id_U$, says that $f =
s_A$. This means that every morphism $f\colon A \to U$ equals $s_A$, so $U$
is terminal and $s_A$ is the unique morphism $!_A$ to the terminal
object $U$. 

The above computation established~(\ref{Opet v Bonnu.}) for $n =
0,1$. Let us continue by induction, assuming that~(\ref{Opet v
  Bonnu.}) was proven for $(n-1)$, with $n \geq 2$. 
The right hand side of~(\ref{Opet v Bonnu.}) is certainly of the form
\[
B_0 \stackrel{g_1} \longrightarrow B_1  \longrightarrow \cdots
\longrightarrow
B_{n-1} \stackrel{g_n} \longrightarrow B_n \longrightarrow T.
\]
Since $d_{n+1}s_{n+1} = \id$,   
\begin{align*}
B_0 \stackrel{g_1} \longrightarrow B_1  \longrightarrow \cdots
\longrightarrow
B_{n-1} \stackrel{g_n} \longrightarrow B_n&  = 
d_{n+1}s_{n+1}(A_0 \stackrel{f_1} \longrightarrow A_1  \longrightarrow \cdots
\longrightarrow
A_{n-1}
\stackrel{f_n} \longrightarrow A_n) 
\\
&=
A_0 \stackrel{f_1} \longrightarrow A_1  \longrightarrow \cdots
\longrightarrow
A_{n-1}
\stackrel{f_n} \longrightarrow A_n.
\end{align*}
Combining the induction assumption with 
$d_0s_{n+1} = s_nd_0$ we obtain that
\begin{align*}
 B_1  \Xrightarrow {g_2} B_2  \longrightarrow \cdots
\longrightarrow
B_{n-1} \stackrel{g_n} \longrightarrow B_n \longrightarrow T 
&=
s_{n}(A_1  \Xrightarrow {f_2} A_2 \cdots
\longrightarrow
A_{n-1}
\stackrel{f_n} \longrightarrow A_n) 
\\
&=
A_1  \Xrightarrow {f_2} A_2 \cdots
\longrightarrow
A_{n-1}
\stackrel{f_n} \longrightarrow A_n \Xrightarrow ! U.
\end{align*}
The above two displays together give~(\ref{Opet v Bonnu.}) for $n$,
proving the induction step.
\end{proof}

\begin{proof}[Proof of Theorem~\ref{Jen aby si stihla vyridit pas.}]
Let $\ttN_\bullet(\C) \cong\dec_\bullet(X)$ and choose as the
operators $d_2$ and $d_3$ in~(\ref{Greta mne spolehlive dovezla.}) the
operators induced by the isomorphism \eqref{NC=decX scheme}. We will show that the condition
in~\eqref{Mozna pojedu az v utery.} is satisfied for this
choice. Since $\C$ is pointed, the terminal object $U$ in~\eqref{Opet
  v Bonnu.} is also initial, so we may chose it to be the zero object
$0 \in \C$. The morphism $\nu_{s_1(f)} : s_1(f) \to s_2(f)$ 
is given by the diagram
\[
\xymatrix{a  \ar[r]^f \ar[d]^{\id_a} & b \ar[r]^{\id_b} \ar[d]^{\id_b} 
& b  \ar[d]^{!}
\\
a \ar[r]^f & b \ar[r]^{!}& 0.
}
\]
The bottom row indeed equals $s_2(f)$ by
Lemma~\ref{Vzal jsem si ze sklepa koste, lopatku ale nemam.}.
The coherence $\psi: d_2d_3 \Rightarrow d_2d_2$  of
$X_\bullet$ par\-ti\-cipates in the diagram
\[
\xymatrix@C=3em{
d_2(d_3(s_1(f)))  \ar[r]^{\psi_{s_1(f)}}\ar[d]_{d_2(d_3(\nu_{s_1(f)}))} 
& d_2(d_2(s_1(f))) \ar[d]^{d_2(d_2(\nu_{s_1(f)}))} \ar@{=}[r]  & d_2(f) \ar[d]^{\id}
\\
d_2(d_3(s_2(f)))  \ar[r]^{\psi_{s_2(f)}}  & d_2(d_2(s_2(f))) \ar@{=}[r]  & d_2(f) \tecka
}
\]
We recall that $\psi_{s_1(f)} = \id$ \/
by~\cite[(1.4.5)]{Jardine} and  $\psi_{s_2(f)} = \id$ \/
by~\cite[(1.4.4)]{Jardine}. The commutativity of the right square
follows from the clear fact that $d_2(\nu_{s_1(f)}) = \id_f$. 
We conclude that $$d_2(d_3(\nu_{s_1(f)})) = \id_{d_2(f)},$$
which is~\eqref{Mozna pojedu az v utery.}. The existence of
kernels thus follows from Proposition~\ref{proposition:kernels_as_functor_II}.

The opposite implication is simple. If $\C$ has kernels, 
it caries a normalized lax $\overline
\Ar$-algebra structure by Theorem~\ref{theorem:kernels}, so it is the
d\`ecalage of an almost strict supercoherent structure by
Corollary~\ref{corollary:Martinuv supervysledek}. This finishes the proof.
\end{proof}

In the rest of this section we show that the isomorphism
$\ttN_\bullet(\C) \cong \dec_\bullet (X)$ imposes strong
conditions on the coherence constraint $\psi$ of $X_\bullet$.

\begin{proposition}
If \/  $\ttN_\bullet(\C) \cong \dec_\bullet (X)$, then the coherence constraints $\psi:
d_{n-1} d_{n} \Rightarrow d_{n-1} d_{n-1}$ are uniquely determined by
the top boundary operators $d_n: X_n \to X_{n-1}$, $n \geq  1$.
\end{proposition}

\begin{proof}
 Replacing
$X_\bullet$ by an isomorphic structure if necessary, 
we will assume that actually 
\begin{equation*}
\label{Lyzuje se dokonce i ve Stromovce a ja trcim v Bonnu!!!}
\ttN_\bullet(\C) = \dec_\bullet (X),
\end{equation*}
i.e.\ that $X_n = \ttN_{n-1}(\C)$, $n \geq 1$.
The central r\^ole in the proof will be played by the transformation
$\nu$ introduced in~\eqref{Na Zimni skole spousta snehu a ja tam
  nebudu.}. We claim that 
\begin{equation}
\label{Mam svitici domecek.}
\psi = d_{n-1}(d_n (\nu)), \ n \geq 2.
\end{equation}
Notice first that,
for $T \in \ttN_{n+1}(\C)$, the bottom row of~\eqref{Na Zimni skole
  spousta snehu a ja tam 
  nebudu.} equals $s_{n+1}(d_{n+1}(T))$ by Lemma~\ref{Vzal jsem si
  ze sklepa koste, lopatku ale nemam.}. As a worm-up, we
analyze in detail~(\ref{Mam svitici domecek.}) in the first nontrivial case
$n=3$. For $T = [f,g] \in \ttN_2(\C)$, $\nu_T$~is the
morphism
\[
\raisebox{-1.7em}{$\nu_T=$} 
\hskip 1em
\xymatrix{
\bullet \ar[r]^{f} \ar[d]^{\id}  & \ar[d]^{\id} \bullet \ar[r]^{g}  
&\ar[d]^{!}  \bullet
\\
\bullet\ar[r]^{f} & \bullet\ar[r]^{!} & 0
}
\]
and the naturality of the transformation $\psi$ implies the
commutativity of
\[
\xymatrix@C=4em{
d_2d_3[f,g] \ar[r]^(.45){\psi_{[f,g]}}  \ar[d]_{d_2d_3(\nu_{[f,g]})}
&    d_2d_2[f,g] = d_2(f)  \ar[d]^{d_2d_2(\nu_{[f,g]})}
\\
d_2d_3s_2d_2[f,g] = d_2(f) 
\ar[r]^{\psi_{s_2 d_2[f,g]}}
&d_2d_2s_2 d_2[f,g] = d_2(f).
}
\]
In this diagram, $\psi_{s_2 d_2[f,g]} = \psi_{s_2(f)}$ is the identity
by~\cite[(1.4.4)]{Jardine}; the equality $d_2d_2(\nu_{[f,g]}) =
\id_{d_2(f)}$ has already been established. The diagram therefore
simplifies to
\[
\xymatrix@C=4em{
d_2(d_3[f,g]) \ar[r]^(.55){\psi_{[f,g]}}  \ar[d]_{d_2d_3(\nu_{[f,g]})}
&  d_2(f)  \ar[d]^{\id}
\\
d_2(f) 
\ar[r]^{\id}
&d_2(f)
}
\]
giving~(\ref{Mam svitici domecek.}) for $n=3$. For a general $T \in
\ttN_{n-1}(\C) = X_n$, $n \geq 3$, we likewise obtain the commutative diagram
\[
\xymatrix{
d_{n-1}d_n(T) \ar[r]^{\psi_T} \ar[d]_{d_{n-1}d_n(\nu_T)} 
& d_{n-1}d_{n-1}(T) \ar[d]^\id
\\                                            
d_{n-1}d_{n-1}(T) \ar[r]^{\id}&  d_{n-1}d_{n-1}(T)\ .
}
\]
We used that 
$\psi_{s_{n-1} d_{n-1}(T)}$ is the identity again, and since $d_{n-1}(\nu_{T})
= \id_{d_{n-1}(T)}$, $d_{n-1}d_{n-1}(\nu_{T})$ is the identity too.
\end{proof}

\begin{proposition}
Assume that $X_\bullet$ is a collection of categories as in
Definition~\ref{definition:almost strict supercoherent}, with the only
difference that we
assume no a priory relations between the functors
\[
d_{n-1}d_n\colon X_n \to X_{n-2} \ \hbox { and }\ d_{n-1}d_{n-1} :X_n \to
X_{n-2}, \ n \geq 2.
\]
Assume moreover that $\dec(X)_\bullet \cong \ttN_\bullet(\C)$ for some
category $\C$. 
Then  $X_\bullet$ is in fact an almost strict
supercoherent structure, i.e.\ a transformations in~\eqref{Vcera
  jsem opet pichl galusku - ty kapitalisticke nic nevydrzi} exist, 
if and only if~\eqref{Mozna pojedu az v
  utery.} holds. If this is the case, then the said
transformations are unique. 
\end{proposition}

\begin{proof}
Formula~\eqref{Mam svitici domecek.} defines natural
transformations $\psi\colon d_{n-1} d_{n} \Rightarrow d_{n-1}
d_{n-1}$, $n \geq 2$, and condition~\eqref{Mozna pojedu az v utery.}
guarantees Jardine's coherence. We leave the details to
the reader.  
\end{proof}

\section{Towards weak operadic categories}
\label{Jarka pojede do Kolumbie, ale Nove Eldorado je legenda.}

Operadic categories, introduced in~\cite{duo}, 
were invented as a natural framework for 
various very general operad-like structures. An operadic category
$\O$ is, by definition, 
equipped with the {\/\em cardinality functor\/}
$|\hbox{-}| : \O \to \Fin$ to the skeletal category of finite sets. 
The operadic category structure
assigns to each 
morphism $f: X \to Y$ in $\O$ an array of $n$ objects of $\O$, 
called the {\em fibers} of $f$, where  $n
\geq 0$ is the cardinality of $Y$. 
We will express the fact that such an $f$ has fibers $F_1,\ldots,F_n$
by writing
\begin{subequations}
\begin{equation}
\label{Uz psano v Praze.}
F_1,\ldots,F_n \ \Fib \ f : X \longrightarrow Y. 
\end{equation} 

We are not going to repeat all the axioms of general operadic categories,
they are easily available not only in the original paper~\cite{duo}, but
also in several follow-up works~\cite{blob,env,lack,GarnerKockWeber,trnka}. 
We also  restrict the following discussion to {\/\em unary\/} operadic
categories, that is, operadic 
categories in which each object has cardinality one. The situation of
general operadic categories is conceptually similar, but involves also
the uninteresting bookkeeping  of multiple fibers. In the unary case, 
(\ref{Uz psano v Praze.}) reduces to much simpler 
\begin{equation}
\label{Bude strasna zima.}
F \  \Fib \ f:  X \longrightarrow Y.
\end{equation}
\end{subequations}
Let us mention that even the simple-minded unary version of operadic
categories covers interesting structures such as decomposition spaces,
$2$-Segal spaces, \&c.

Item~(iii) of the axioms in~\cite[page~1634]{duo} requires that
the fiber assignment $f \mapsto F$ in~(\ref{Bude strasna zima.}) 
be functorial, meaning that it assembles to 
the {\/\em fiber functor\/}
\[
\F : \Dec(\O) \longrightarrow \O,
\]
from the d\' ecalage~\eqref{Dnes jsme byli s Jarkou na vystave betlemu.}
of the category $\O$. Item (iv) of the operadic axioms
then requires the identity
$X = Y$ of the fibers in the diagram 
\begin{equation}
\label{Pozitri pojedeme na chalupu.}
\xymatrix@C=1.1em@R=1em{X \ar@{}[r]|{\hbox{\large $\triangleright$}}  
\ar@{=}[d] 
&\ar[rr]^{\F(S)} \ar@{}[d]|\spicka F  &&G \ar@{}[d]|\spicka  
\\
Y \ar@{}[r]|{\hbox{\large $\triangleright$}} &A  \ar[rr]^f  
\ar@/_/[ddr]_{gf}  &\ar@{}[dd]|{\raisebox{.3em}{\hbox{$S$}}}&B \ar@/^/[ddl]^g
\\
\\
&&C,&
}
\end{equation}
in which $\F(S)$ is the value of the fiber functor on the map $S : gf
\to f$ in $\D(\O)$.
This requirement is perfectly reasonable and justified 
in the applications
for which operadic categories were invented. However, there are situations 
in which one wishes to replace
the equality $X=Y$ by the existence of a {\em natural} isomorphism 
$\phi_S: X \to Y$ so that~(\ref{Pozitri pojedeme na chalupu.}) 
is replaced by 
\begin{equation}
\label{Pozitri pojedeme na chalupu.2}
\xymatrix@C=1.1em@R=1em{X \ar@{}[r]|{\hbox{\large $\triangleright$}}  
\ar[d]_{\phi_S} 
&\ar[rr]^{\F(S)} \ar@{}[d]|\spicka F  &&G \ar@{}[d]|\spicka  
\\
Y \ar@{}[r]|{\hbox{\large $\triangleright$}} &A  \ar[rr]^f  
\ar@/_/[ddr]_{gf}  &\ar@{}[dd]|{\raisebox{.3em}{\hbox{$S$}}}&B \ar@/^/[ddl]^g
\\
\\
&&C.&
}
\end{equation}
This situation arises, for example, when $\O$ is an abelian category with 
fibers given by kernels, cf.~\cite[Example~1.22]{duo}. 
Then $\phi_S$ uniquely determined by the
universal property of the kernels, in the same way as in
the second part of the proof of Theorem~\ref{theorem:kernels}. 
Also, an isomorphism is more pleasing to a categorist’s eye than
equality. The problem, of
course, is to say precisely what ``naturality of $\phi_S$'' means.

As noticed in~\cite{GarnerKockWeber}, the d\'ecalage functor $\Dec$ 
is a comonad in $\Cat$, with the
commultiplication $\delta: \Dec(\C) \to \Dec^2(\C)$ given by
\[
\raisebox{-2em}{$\delta(f:A\to B) := $} \hskip 1.5em
\xymatrix@C=1.1em@R=1em{
A  \ar[rr]^f  
\ar@/_/[ddr]_{f}  &&B \ar\ar@/^/[ddl]^{\id_B}
\\
\\
&B&
}
\]
and the counit $\eta: \Dec(\C) \to \C$   
given by the domain functor.
The comonad $\Dec$ induces in the
standard manner, a monad  $\overline{\Dec}$ on the category
$\Cat_\Dec$ of $\Dec$-coalgebras. Theorem~10 of~\cite{GarnerKockWeber} 
then states that 
$\overline{\Dec}$-algebras are precisely unary
operadic categories. A tempting idea is thus to mimic the scheme applied
on the arrow comonad $\Ar$ in the previous sections and {\em define} lax
operadic categories as lax $\overline{\Dec}$-algebras. The diagram
in~(\ref{Pozitri pojedeme na chalupu.2}) will then by expressed by the lax
associativity, cf.~the diagram in~\eqref{V Mercine bude pekna zima.}, 
of the functor $\F : \Dec(\C) \to \C$. 

The  caveat is that it makes sense to speak about lax monadic algebras
only for $2$-monads, otherwise the induced natural transformation 
$\T\phi$ in (ii) of Definition~\ref{definition:lax_algebra} 
is not defined. However, $\Dec$ is, in contrast with the arrow comonad
$\Ar$, only an ordinary comonad, and therefore the induced monad
$\overline{\Dec}$ is also an ordinary monad. Indeed, neither the formula
in~\eqref{Na koncert jede cely sbor vlakem} nor any similar formula
makes sense for $\Dec$. So an ``invisible hand'' leads us to the use of
$\Ar$ instead of $\Dec$. 

Another option could, in  principle, be based on the approach that
characterizes unary operadic categories as those whose simplicial 
nerve $\ttN_\bullet(\C)$ is a
d\'ecalage of a simplicial set $X_\bullet$, cf.~\cite[Section~1]{blob}.
The fiber functor
then materializes in the form of the additional simplicial operators
\begin{equation}
\label{Zabalil jsem ruzne.}
d_{n+1}: \ttN(\C)_n \to \ttN(\C)_{n-1},\ n \geq 1.
\end{equation}
A lax version could then be obtained by replacing the
`strict' simplicial set $X_\bullet$ by a supercoherent structure as
done in Section~\ref{Jarka si koupila novy mobil.} of the present article. 
However, supercoherence requires categorification, so the
operators~(\ref{Zabalil jsem ruzne.}) must be functors on 
the components of the categorified nerve. For $n=1$, this means that
$d_2$, which is the actual fiber functor, acts functorially on the squares
\[
\xymatrix{a \ar[r]^{h_0} \ar[d]_{f}  & c \ar[d]^{g}
\\
b  \ar[r]^{h_1} & d
}\
\]
representing a morphism $(h_0,h_1) \colon f \to g$ in the category
$\ttN_1(\C) = \Ar(\C)$. This brings us back to the arrow category.

\noindent 
{\bf Conclusion.} 
A na\"\i ve approach to lax operadic categories does
not work, since it necessarily implies an 
extended functoriality of the fiber functor. 
The closest notion that could be called a lax unary operadic
category is a lax algebra over the arrow category monad
$\overline{\Ar}$, 
equivalent to a system of kernels. General non-unary lax 
operadic categories should
appear as a generalization of this concept.


\bibliographystyle{plain}

\end{document}